\documentclass[12pt]{article}
\usepackage[top=1in, bottom=1in, left=0.5in, right=0.5in]{geometry}
\usepackage{hyperref}
\usepackage{amssymb}
\usepackage{amsmath}
\usepackage{amsthm}
\usepackage{mathrsfs}
\usepackage{graphicx}
\hypersetup{
    colorlinks=true,
    linkcolor=blue,
    filecolor=magenta,      
    urlcolor=cyan,
    pdftitle={Overleaf Example},
    pdfpagemode=FullScreen,
    }

\numberwithin{equation}{section}
\newtheorem{theorem}{Theorem}[section]
\newtheorem{corollary}{Corollary}[theorem]
\newtheorem{lemma}[theorem]{Lemma}
\newtheorem{remark}[theorem]{Remark}
\numberwithin{equation}{section}
\title{\textbf{\Large Higher order dualities between prime ideals}}

\author{\textit{Sroyon Sengupta}}
\date{\textit{Dedicated to  Prof. Krishnaswami Alladi on his $70^{th}$ Birthday}}

\begin{document}

\maketitle
\noindent {\small \textbf{Abstract:} \textit{Extending the works of Alladi [Al77] and Sweeting and Woo [SW18], we state and prove the general higher order duality between prime ideals in number rings. We then use the second order duality to obtain a new formula for the Chebotarev Density involving sums of the generalized Möbius function and the prime ideal counting function. We also provide two estimates of such sums as an application of the duality identity. A discussion of the duality in a slightly more general setting is done at the end. }} 

\section{Introduction and Background}
Since the $19^{th}$ century, the M\"obius function has been an integral and essential part of analytic number theory, and several important results and theories have come up around it. Let us first recall this significantly  important function defined on positive integers:
\begin{align}
    \mu(n) = \begin{cases}
        1,\quad \text{when $n=1$}, \\
        0,\quad \text{when $n$ is not square-free}, \\
        (-1)^k,\quad \text{when $n=p_1p_2...p_k$, where $p_i$'s are distinct prime numbers.}
    \end{cases} \nonumber
\end{align}
One of the most well-known examples of the use of the M\"obius function is in its relation with the Riemann Zeta function. Using the Euler product formula of $\zeta(s)$, one can show that the Dirichlet series, with $\mu(n)$ as its coefficients, is the reciprocal of the zeta function. Further, 
\begin{equation}
    \sum_{n=1}^{\infty}\frac{\mu(n)}{n} = \lim_{s\to 1^+}\frac{1}{\zeta(s)} = 0. \nonumber
\end{equation}
Landau [LaT99] proved that 
\begin{align}
    \sum_{n=1}^{\infty}\frac{\mu(n)}{n} = 0
\end{align}
is (elementarily) equivalent to the Prime Number Theorem. A surprising generalization of (1.1) was observed by Alladi [Al77] in 1977, as a consequence of the following Duality Lemma he proved in the same paper. He proved that if $f$ is a function defined on primes, then:
\begin{align}
    \sum_{1<d|n}\mu(d)f(P_1(d)) &= -f(p_1(n)), \nonumber \\
    \sum_{1<d|n} \mu(d)f(p_1(d)) &= -f(P_1(n)),
\end{align}
where $P_1(n)$ and $p_1(n)$ are the largest and the smallest prime factors of the integer $n$, respectively. Using (1.2) and some properties of the Mobius function, Alladi showed that if $f$ is a bounded function on the primes such that
\begin{align}
    \lim_{x\to \infty}\frac{1}{x}\sum_{2\leq n\leq x}f(P_1(n)) = c,    
\end{align}
then 
\begin{align}
    \sum_{n=2}^{\infty} \frac{\mu(n)f(p_1(n))}{n}=-c,
\end{align}
and vice versa. We notice that (1.1) is indeed a special case of (1.4). If we choose $f$ to be an arithmetic function on primes defined as $f(p)=1$ for all primes $p$, then such an $f$ will yield from (1.3) that 
\begin{align}
    \lim_{x\to \infty}\frac{1}{x}\sum_{2\leq n \leq x}f(P_1(n)) = \lim_{x\to \infty}\frac{[x]-1}{x} = 1. \nonumber
\end{align}
Then, from (1.4), we get for the same $f$ that
\begin{align}
    \sum_{n=2}^{\infty}\frac{\mu(n)}{n} = -1, 
\end{align}
which is nothing but (1.1) with a slight rearrangement. \\ \\
Another example of $f$ is the characteristic function of primes in an arithmetic progression $\ell\;(mod\;k)$. In [Al77], it was proved that the Prime Number Theorem in Arithmetic Progressions implies that the sequence $\{P_1(n)\}_n$ of the largest prime factors of integers $n$ is uniformly distributed in the reduced residue classes modulo a positive integer, say $k$. This further implies that, with such a choice of $f$ mentioned right above, the average of $f(P_1(n))$ exists, and is $\frac{1}{\varphi(k)}$. Therefore, the value of $c$ in (1.3) for this corresponding choice of $f$ is $\frac{1}{\varphi(k)}$ and hence, by (1.4),
\begin{align}
    \sum_{\substack{n=2 \\ p_1(n)\equiv\ell \;(mod\;k)}}^{\infty}\frac{\mu(n)}{n} = -\frac{1}{\varphi(k)},
\end{align}
for all positive integers $k,\ell$, such that $1\leq \ell<k$ and $(\ell,k)=1$. We note here that the sum on the left of (1.6) gives us a way of slicing the sum on the left of (1.5) into $\varphi(k)$ subseries, each of them converging to the same value! Alladi gave a quantitative proof of (1.6) in [Al77], which involves the use of the duality identity (1.2). \\ \\
In 2017, Dawsey [Da17] chose $f$ to be the characteristic function of primes satisfying the condition $\left[\frac{K/\mathbb{Q}}{p}\right]=C$, where $K$ is a finite Galois extension of the field of rationals $\mathbb{Q}$ and $C$ is a fixed conjugacy class of the Galois group $G = Gal(K/\mathbb{Q})$. She first showed that (see \textit{Theorem 2} in [Da17]) that with the choice of the above $f$, the density of the sequence $\{f(P_1(n))\}_n$ for integers $n$ is $\frac{|C|}{|G|}$, i.e. the constant $c$ in (1.3) is $\frac{|C|}{|G|}$. Hence, by (1.4) again, we get (see \textit{Theorem 1} [Da17]) 
\begin{align}
     \sum_{\substack{n=2 \\ \left[\frac{K/\mathbb{Q}}{p_1(n)}\right]=C}}^{\infty}\frac{\mu(n)}{n} = -\frac{|C|}{|G|}.
\end{align}
What stands out in [Da17] is the quantitative nature of the results proved by Dawsey, which she is able to derive using the strong form of the Chebotarev Density Theorem given by Lagarias-Odlyzko [LO77]. It is important to note here that the setting in which (1.7) has been proved considers the base field of the Galois extension to be fixed as the field of rationals, although the choice of $K$ is ``almost"\footnote[1]{``almost", since $K$ must abide by the condition that $K/\mathbb{Q}$ is finite and Galois} arbitrary. \\ \\
Extending the previous works by Alladi and Dawsey, in 2018, Sweeting and Woo [SW18] proved similar density-type results by making the choice of the base field of the Galois extension to be any arbitrary number field. In particular, if $C\subset G=Gal(L/K)$ is a conjugacy class, then they proved the following new formula for the Chebotarev Density $\frac{|C|}{|G|}$:
\begin{align}
    -\lim_{X\to \infty} \sum_{\substack{2 \leq N(I)\leq X \\ I \in S(L/K;C)}} \frac{\mu_K(I)}{N(I)} = \frac{|C|}{|G|},
\end{align}
where $\mu_K(I)$ denotes the generalized M\"obius function (see below) and $S(L/K;C)$ is the set of ideals $I \in \mathcal{O}_K$ such that $I$ has a unique prime divisor $\mathfrak{p}_1$ of minimum norm and the Artin sumbol $\left[\frac{L/K}{\mathfrak{p}}\right]$ is $C$. We would like to note here that proving (1.8) required new machinery, as Alladi's Duality lemma is only valid for the classical case. Since Dawsey proves her results by considering the base field to be $\mathbb{Q}$, Alladi Duality works fine as it is applicable to integers, which act like the ideals in $\mathcal{O}_{\mathbb{Q}}$, i.e., $\mathbb{Z}$. Hence, Sweeting and Woo prove the following first order duality between prime ideals (see \textit{Lemma 2.1} in [SW18]): if $f(I)$ denotes the indicator function of $S(L/K;C)$, then for ideals $I$, we have
\begin{align}
    \sum_{J \supset I}\mu_K(J)f(J) =-Q_C(J),
\end{align}
where the definition of the quantity in the RHS above is given by
\begin{align}
    Q_C(I) = \#\left\{I \subset \mathfrak{p}:N(\mathfrak{p})=M(I)\;\text{and}\;\left[\frac{L/K}{\mathfrak{p}}\right]=C\right\},
\end{align}
where $M(I) = \max_{I \subset \mathfrak{p}}N(\mathfrak{p})$. This duality between the prime ideals with the smallest and largest norms proved crucial in proving (1.8). Not only that, the proof of (1.8) is also quantitative in nature, which makes it even more interesting. Several other results in [SW18] are extensions of those in both [Al77] and [Da17], hence connecting the ideas of both the classical case and the algebraic case. The most important one to take note of here is \textit{Theorem 1.1}, where they use the strong form of CDT. One can indeed find a correspondence between the same and the uniform distribution of the largest prime factors in an arithmetic progression $\ell\;(mod\;k)$ [Al77] and the existence of the density of $f(P_1(n))$, where $f$ is the characteristic function of the prime numbers satisfying $\left[\frac{K/\mathbb{Q}}{\mathfrak{p}}\right]=C$ [Da17]. \\ \\
In [Al77], Alladi also proved a general duality lemma concerning the $k^{th}$ smallest or the $k^{th}$ largest prime factors of integers accordingly. In 2024, he and Johnson used a version of his second order duality to prove that (see \textit{Theorem 13} [AJ24]) for a function $f$ that is bounded on primes such that
\begin{align}
    \sum_{2 \leq n \leq x}f(P_1(n)) \sim \kappa x \quad \text{and} \quad \sum_{2 \leq n \leq x} f(P_2(n)) \sim \kappa x,
\end{align}
then
\begin{align}
    \sum_{n=2}^{\infty} \frac{\mu(n)f(p_1(n))\omega(n)}{n}=0.
\end{align}
We note here that $\omega(n)$ is defined by the number of distinct prime factors of $n$ and $P_2(n)$ is defined to be the second largest prime factor of $n$. In both [AJ24] and [Se25], there have been discussions on the definitions of the second largest prime factors. Following them, $P_2(n)$ here is defined by the largest prime factor of $n$ strictly less than $P_1(n)$. Their paper itself gives us the two examples that help us understand the above \textit{Theorem 13}. More enchantingly, both of those examples have been proved as theorems (see \textit{Theorem 4(i)} and \textit{Theorem 10} in [AJ24]), and the proofs are quantitative in nature. \\ \\
For the choice of a function $f$ defined on primes such that $f(p)=1$ for all primes, the constant $\kappa$ in (1.11) exists and is 1, and therefore,
\begin{align}
    \sum_{n=2}^{\infty}\frac{\mu(n)\omega(n)}{n} =0
\end{align}
holds. This is exactly the qualitative form of \textit{Theorem 4(i)}. Further, the choice of $f$ to be the characteristic function of primes in the arithmetic progression $\ell\;(mod\;k)$ gives us that $\kappa = \frac{1}{\varphi(\ell)}$ (see \textit{Theorem 7} [AJ24]) and for such co-prime integers $\ell,k$,
\begin{align}
    \sum_{\substack{n=2 \\ p_1(n)\equiv \ell\;(mod\;k)}}^{\infty}\frac{\mu(n)\omega(n)}{n} =0.
\end{align}
The author of this paper, following in the lines of Dawsey, proved recently [Se25] that the choice of $f$ to be the characteristic function of primes satisfying the Artin symbol condition $\left[\frac{K/\mathbb{Q}}{p}\right]=C$ yields $\kappa=\frac{|C|}{|G|}$ (see \textit{Theorem 2.1} [Se25]), i.e. the Chebotarev Density, and therefore, proves
\begin{align}
     \sum_{\substack{n=2 \\ \left[\frac{K/\mathbb{Q}}{p}\right]=C}}^{\infty}\frac{\mu(n)\omega(n)}{n} =0.
\end{align}
Of course, [Se25] gives a quantitative proof of (1.15) (see \textit{Theorem 2.5}) using the strong form of CDT. An arithmetic density version was also noted as follows: 
\begin{align}
     \sum_{\substack{n=2 \\ \left[\frac{K/\mathbb{Q}}{p}\right]=C}}^{\infty}\frac{\mu(n)(\omega(n)-1)}{n} = \frac{|C|}{|G|}, \nonumber
\end{align}
providing us a new formula for the Chebotarev Density. \\ \\
Like Dawsey, [Se25] also extends results in [AJ24] to the setting of finite Galois extensions with the base field chosen to be the field of rationals $\mathbb{Q}$. With the work of Sweeting and Woo in the picture as discussed above, it is quite natural to study the generalization of the results of Alladi and Johnson, and the author himself, to the setting of arbitrary finite Galois extensions. In that case, too, one would need a new version of duality between prime ideals that involves prime ideals of $k^{th}$ largest norms. In this paper, we first state and prove (see \S 2, \textit{Theorem 2.4}) the higher order duality lemma between prime ideals. We then use the second order duality (see \textit{Theorem 2.3}) specifically to prove the density type extensions\footnote[2]{new formula for the Chebotarev Density} of results in [AJ24] and [Se25]. Sections that will follow the proofs of our main theorems will also provide applications of this new general duality between prime ideals in number fields and remark on how it can prove to be essential in generalizing all such duality-related results that have been derived by several other researchers (see \S5,\S6) in the classical case. For now, let us look at the necessary notations and preliminary theorems that we will use in the proofs to come. \\ \\
Throughout the rest of the paper, $L/K$ will denote a finite Galois extension of number fields such that $G= Gal(L/K)$. The corresponding rings of integers of the fields involved in this extension will be denoted as $\mathcal{O}_K$ and $\mathcal{O}_L$. Let $\mathfrak{p} \in \mathcal{O}_K$ denote a prime ideal in $\mathcal{O}_K$. Then the ideal generated by $\mathfrak{p}$ in $\mathcal{O}_L$ can be decomposed into a product of distinct prime ideals $\mathfrak{p}_i$'s in $\mathcal{O}_L$ lying above $\mathfrak{p}$, i.e. 
\begin{align}
    \mathfrak{p}\cdot\mathcal{O}_L = \prod_{i=1}^k \mathfrak{p}_i^{e_i}, 
\end{align}
where $e_i$'s are positive integers. Our only concern, in the results that follow, is prime ideals in $\mathcal{O}_K$ that are unramified in $\mathcal{O}_L$, i.e., the corresponding integers $e_i$, for each $i$, are 1. We further recall that the absolute norm of a non-zero ideal $\mathcal{I}$ in any number ring $\mathcal{O}$ is given by
\begin{align}
    N(\mathcal{I}):=[\mathcal{O}:\mathcal{I}] = |\mathcal{O}/\mathcal{I}|. \nonumber
\end{align}
We now introduce the definition of the Artin symbol corresponding to a prime ideal $\mathfrak{p}$ in $\mathcal{O}_K$. For a given prime ideal $\mathfrak{p} \subset \mathcal{O}_K$, unramified in $\mathcal{O}_L$, let us consider the decomposition noted above in (1.16) with each $e_i=1$. For every prime ideal $\mathfrak{p}_i$ in (1.16), the Artin symbol $\left[\frac{L/K}{\mathfrak{p}_i}\right]$ is defined as the unique\footnote[3]{the unramification of $\mathfrak{p}$ is necessary to ensure this uniqueness} isomorphism $\sigma \in G$ satisfying the condition
\begin{align}
    \sigma(\alpha) \equiv \alpha^{N(\mathfrak{p})}\;(mod\;\mathfrak{p}_i), \nonumber
\end{align}
for every $\alpha \in L$. Now, for any member of $G$, say $\tau$, the ideals $\mathfrak{p}_i$'s are isomorphic to each other under $\tau$ in a particular permutation, depending on the definition of $\tau$. Moreover, for such $\tau$, we have
\begin{align}
    \left[\frac{L/K}{\tau(\mathfrak{p}_i)}\right] = \tau\left[\frac{L/K}{\mathfrak{p}_i}\right]\tau^{-1}, \nonumber
\end{align}
holds. Therefore, we get a conjugacy class $C \subset G$ associated to the prime ideal $\mathfrak{p}$. We then define the Artin Symbol $\left[\frac{L/K}{\mathfrak{p}}\right]$  to a prime ideal $\mathfrak{p}$ in the ring of integers $\mathcal{O}_K$ to be the associated conjugacy class $C$. \\ \\
Let us now denote the set of all prime ideals in $\mathcal{O}_K$ by $\mathfrak{P}(K)$. Given a conjugacy class $C \subset G$, we define
\begin{align}
    \mathfrak{P}_C : = \left\{\mathfrak{p} \in \mathfrak{P}(K):\;\mathfrak{p} \text{ is unramified in $L$ and } \left[\frac{L/K}{\mathfrak{p}}\right]=C \right\}. \nonumber
\end{align}
We denote $\pi_C(X,L/K)$ to be the number of non-zero prime ideals $\mathfrak{p} \subset \mathcal{O}_K$ that are unramified in $L$ with $N(\mathfrak{p}) \leq X$ and their corresponding Artin Symbol to be a fixed conjugacy class $C$.
\begin{theorem}{(Chebotarev Density Theorem [Ts26])}
    Let $L/K$ be a finite Galois extension and $C$ be a conjugacy class of the Galois group $G=Gal(L/K)$. Then the natural density of $\mathfrak{P}_C$ defined by 
    \begin{align}
        \lim_{x\to \infty} \frac{\#\{\mathfrak{p} \in \mathfrak{P}_C:N(\mathfrak{p}) \leq X\}}{\#\{\mathfrak{p} \in \mathfrak{P}(K): N(\mathfrak{p}) \leq X\}} \nonumber
    \end{align}
    exists and is equal to the ratio $\frac{|C|}{|G|}$. More precisely, as $x \to \infty$
    \begin{align}
        \pi_C(X,L/K) = \frac{|C|}{|G|}\cdot \frac{X}{\log X}+o\left(\frac{X}{\log X}\right). \nonumber
    \end{align}
\end{theorem}
\noindent In our proofs, we will use a stronger form of the above \textit{Theorem 1.1} with a more precise formulation of the error given by Lagarias and Odlyzko [LO77] as follows:
\begin{theorem}
    (Stronger form of CDT [LO77]) For sufficiently large $X \geq \Tilde{C}$, where $\Tilde{C}$ depends on both the absolute discriminant and the degree of extension of $L$ (over $\mathbb{Q}$), we have that 
    \begin{align}
        \left|\pi_C(X,L/K)-\frac{|C|}{|G|}Li(X)\right| \ll Xe^{-c_1\sqrt{\frac{\log x}{n_L}}}, \nonumber
    \end{align}
    for an appropriate constant $c_1$, where $n_L=[L:\mathbb{Q}]$ and $Li(x) = \int_2^x\frac{dt}{\log t}$.
\end{theorem}
\noindent It is important to note the error estimate above in \textit{Theorem 1.2} as it is a crucial element in the proofs to come. We also note Landau's form of the Prime Ideal Theorem [La03], which will also be essential later on. We denote $\pi(K;X)$ to be the number of prime ideals in $\mathcal{O}_K$ whose norm is bounded by $X$. 
\begin{theorem}
    (Prime Ideal Theorem [La03]) If $K$ is a number field, then there exists a positive constant $c_2$ such that for large enough $X$, we have
    \begin{align}
        |\pi(K;X)-Li(X)| \ll Xe^{-c_2\sqrt{\log X}} .\nonumber
    \end{align}
\end{theorem}
\noindent Note here that $c_1=c_2$ as if since considering $K = \mathbb{Q}$, \textit{Theorem 1.3} directly follows from \textit{Theorem 1.2}.
\begin{remark}
    Theorem 1.3 is the more explicit version of the Prime Ideal Theorem by Landau. It is also worthwhile to note the asymptotic version of the same (see also [MV07], p. 194), i.e., if $K$ is a number field then 
    \begin{align}
        \pi(K;X) \sim \frac{X}{\log X} .\nonumber
    \end{align}
    We will be using this version quite often in our proofs.
\end{remark}
\noindent Another important theorem that gives us the count of the ideals of bounded norm will also be needed later - the following is Murty and Van Order's [MO07] explicit version of a classical asymptotic for the number of ideals.
\begin{theorem}
    As $X \to \infty$, we have
    \begin{align}
        \sum_{N(I)\leq X}1 = c_K\cdot X +O\left(X^{1-\frac{1}{d}}\right), \nonumber
    \end{align}
    where $c_K$ is the residue of the simple pole of the Dedekind Zeta function $\zeta_K(s)$ at $s=1$, given by
    \begin{align}
        c_K:= \frac{2^{r_1}\cdot (2\pi)^{r_2}\cdot Reg_K\cdot h_K}{w_K\cdot \sqrt{|D_K|}}, \nonumber
    \end{align}
    where $r_1$ and $2r_2$ are the number of real and complex embeddings of $K$ respectively, $h_K$ is the class number, $Reg_K$ is the regulator, $w_K$ is the number of roots of unity in $K$ and $D_K$ is the discriminant.
\end{theorem}
\noindent In the next section, we will state and prove the general higher order duality between prime ideals. We will first look at the second order duality to understand how the proof works and then move ahead with the general proof. Note that Sweeting and Woo proved the first order duality (see \textit{Lemma 2.1} in [SW18]), and we will see how it is consistent with our general duality identities. In the sections that follow \S2, we use the second order duality to get a new formula for the Chebotarev Density, with its quantitative proof. 

\section{General Duality Between Prime Ideals} 
We start by stating the general duality lemma in the classical case given by Alladi (see \S 4 [Al77]). Let $P_k(n)$ and $p_k(n)$ denote the $k^{th}$ largest and $k^{th}$ smallest prime factors of an integer $n$, respectively. In the cases where $n=1$, or for an integer $n$ such that $\omega(n)<k$, we set $P_k(n)=1=p_k(n)$.  
\begin{lemma}
    If $f$ is an arithmetic function with $f(1)=0$, then the following four identities hold:
    \begin{align}
        \sum_{d|n} \mu(d)f(P_k(d))=(-1)^k {\omega(n)-1 \choose k-1}f(p_1(n)), \nonumber\\
    \sum_{d|n} \mu(d)f(p_k(d))=(-1)^k {\omega(n)-1 \choose k-1}f(P_1(n)), \nonumber
    \end{align}
    and
    \begin{align}
         \sum_{d|n} \mu(d){\omega(d)-1 \choose k-1}f(P_1(d))=(-1)^k f(p_k(n)), \nonumber\\
    \sum_{d|n} \mu(d){\omega(d)-1 \choose k-1}f(p_1(d))=(-1)^k f(P_k(n)) .\nonumber
    \end{align}
\end{lemma}
\noindent As one can observe, this is applicable only for integers, and therefore, we need and will proceed to state and prove the duality version for prime ideals in arbitrary number fields. Let us first define a few quantities and sets essential for stating the duality identities. Let $N(I)$ denote the norm of an ideal $I$. For an ideal $I \subset \mathcal{O}_K$, we define $M_k(I)$ to be the norm of the prime ideals in the decomposition of the ideal $I$ with the $k^{th}$ largest norms. Of course, the quantity $M_k(I)$ makes sense only if $I$ is contained in prime ideals with $k$ many distinct norms. Note that there might not be a unique prime ideal containing $I$ with a particular norm.  In view of this fact, we define the sets $Q^k(I)$ and $Q_C^k(I)$ as follow: for a finite Galois extension $L/K$ and given a fixed conjugacy class $C \subset G = Gal(L/K)$, 
\begin{align}
     Q_C^k(I) &:= \# \left\{\mathfrak{p} :\; I\subset \mathfrak{p}; \;M_k(I)=N(\mathfrak{p})\;\text{and}\;\left[\frac{L/K}{\mathfrak{p}}\right]=C,\;\mathfrak{p}\;\text{is unramified}\right\}, \nonumber \\ 
     Q^k(I)&:=\# \{\mathfrak{p}:\;I \subset \mathfrak{p};\;M_k(I)=Nm(\mathfrak{p})\} . \nonumber
\end{align}
From now onwards, we only focus on such ideals $I$ that have a unique prime ideal of the smallest norm. For a fixed $I$ satisfying such a property, we denote the unique prime factor of $I$ with the smallest norm as $\mathfrak{p}_1$. Following the convention proposed in [SW18], we also call these ideals the \textit{salient ideals} in $\mathcal{O}_K$. Let us define the set $S(L/K;C)$ as follows: 
\begin{align}
     S(L/K;C):=\left\{I \subset \mathcal{O}_K:\;I\text{ is salient and }\left[\frac{L/K}{\mathfrak{p}_1}\right]=C,\;\text{where $\mathfrak{p}_1$ is unramified}\right\}. \nonumber
\end{align}
We now recall the definition of the generalized M\"obius function. For a number field $K$, we denote $\mu_K$ to be the generalized M\"obius function which is defined, for an ideal $I \subset \mathcal{O}_K$, as
\begin{align}
     \mu_K(I) = \begin{cases}
        1,\quad \text{when $I=\mathcal{O}_K$}, \\
        0,\quad \text{when $I \subset \mathfrak{p}^2$ for some prime ideal $\mathfrak{p}$}, \\
        (-1)^k,\quad \text{when $n=\mathfrak{p}_1\mathfrak{p}_2...\mathfrak{p}_k$, where $\mathfrak{p}_i$'s are distinct prime ideals.}
    \end{cases} \nonumber
\end{align}
The following is \textit{Lemma 2.1} in [SW18]\footnote[4]{we can view the identity to be the first order duality between prime ideals once we have stated the general case}:
\begin{lemma}
    Let $f(I)$ be the indicator function of $S(L/K;C)$. Then for ideals $I$, we have
    \begin{align}
        \sum_{J \supset I}\mu_K(J)f(J) = -Q_C(I). \nonumber 
    \end{align}
\end{lemma}
\noindent Note that as per our definitions, $Q_C(I)$ and $Q_C^1(I)$ denote the same set.
We further use $\omega_K(I)$ to denote the prime ideal counting function 
\begin{align}
    \omega_K(I) = \sum_{\mathfrak{p} \supset I} 1 .\nonumber 
\end{align}
We first state and prove what we view as the second order duality between prime ideals, i.e., duality for $k=2$. 
\begin{theorem}
    Let $K$ be a number field and $I \subset \mathcal{O}_K$ be an ideal. Let $f$ be the characteristic function of the set $S(L/K;C)$. Then, for ideals satisfying $Q^1(I)=1$\footnote[5]{this means that such ideals have a unique prime ideal factor with the maximum norm}, the following identity holds:
    \begin{align}
         \sum_{J \supset I} \mu_K(J)(\omega_K(J)-1)f(J) = Q_C^2(I). \nonumber
    \end{align}
\end{theorem}
\begin{proof}
    We start with a fixed ideal $I$ satisfying the required hypothesis, such that $I = \mathfrak{p}_1^{e_1}\mathfrak{p}_2^{e_2}...\mathfrak{p}_r^{e_r}$. Without the loss of generality, we assume that in the factorization above, $N(\mathfrak{p}_i)\leq N(\mathfrak{p}_{i+1})$, that is, we factor $I$ into prime ideals in an ascending order of their norms. We rewrite the factorization as $I=I_1I_2...I_s$, where each of these $I_j$'s is a product of the respective prime ideals of the same norm. For example, if $\mathfrak{p}_{i_1}^{e_{i_1}}\mathfrak{p}_{i_2}^{e_{i_2}} \supset I_i$, for some integers $i,i_1,i_2$, then $N(\mathfrak{p}_{i_1})=N(\mathfrak{p}_{i_2})$. Therefore, 
    \begin{align}
        &\sum_{J \supset I} \mu_K(J)(\omega_K(J)-1)f(J) \nonumber \\
        =& \sum_{j=1}^{s-1} \sum_{\substack{J \supset I_j \\ J \not \supset I_{k'},\;\forall k'\geq j+1 \\ J \neq \mathcal{O}_K}}\mu_K(J)f(J)\sum_{J' \supset I_{j+1}...I_s} \mu_K(J')(\omega_K(JJ')-1) \nonumber  \\ 
        =& \sum_{j=1}^{s-1} \sum_{\substack{J \supset I_j \\ J \not \supset I_{k'},\;\forall k'\geq j+1 \\ J \neq \mathcal{O}_K}}\mu_K(J)f(J)\sum_{J' \supset I_{j+1}...I_s} \mu_K(J')(\omega_K(J)-1+\omega_K(J')).
    \end{align}
    We note that in the above equality, the outer sum on the right ranges to $s-1$, since the contribution from the ideal $I_s$ is 0. Also, we do not need the term $f(J')$ as by the definition of $f$, we only need to focus on the ideal with the smallest norm, which, of course, contains $J$ and not $J'$. Therefore, continuing from (2.1), we have
\begin{align}
    \sum_{J \supset I} \mu_K(J)(\omega_K(J)-1)&f(J) = \sum_{j=1}^{s-1} \sum_{\substack{J \supset I_j \\ J \not \supset I_{k'},\;\forall k'\geq j+1 \\ J \neq \mathcal{O}_K}}\mu_K(J)f(J)(\omega_K(J)-1)\sum_{J' \supset I_{j+1}...I_s} \mu_K(J') \nonumber \\
    &+ \sum_{j=1}^{s-1} \sum_{\substack{J \supset I_j \\ J \not \supset I_{k'},\;\forall k'\geq j+1 \\ J \neq \mathcal{O}_K}}\mu_K(J)f(J)\sum_{J' \supset I_{j+1}...I_s} \mu_K(J')\omega_K(J') =:\mathcal{S}_1+\mathcal{S}_2.
\end{align}
We now evaluate both the summations $\mathcal{S}_1$ and $\mathcal{S}_2$ above in (2.2). Starting with the inner sum of $\mathcal{S}_1$, we have that
\begin{align}
    \sum_{J' \supset I_{j+1}...I_s}\mu_K(J') = \sum_{i=0}^t (-1)^i{t \choose i},
\end{align}
where we assume that there are $t$-many distinct prime ideals containing the product of ideals $I_{j+1}...I_s$. 
On the other hand, we know that 
\begin{align}
    (1-x)^t = \sum_{i=0}^t (-1)^i{t \choose i} x^i.
\end{align}
Hence, the RHS of (2.4) with $x=1$ yields the RHS of (2.3), and therefore, we have from both equations that
\begin{align}
     \sum_{J' \supset I_{j+1}...I_s}\mu_K(J') = \begin{cases}
         1,\;\text{when $I_{j+1}...I_s = \mathcal{O}_k$} \\
         0,\;\text{otherwise.}
     \end{cases}
\end{align}
From (2.5), it is clear that the only contribution from $\mathcal{S}_1$ is when $I_{j+1}...I_s = \mathcal{O}_K$, which is again only possible if $j=s$. But that is not possible since $j$ ranges from $1$ to $s-1$, and hence, from (2.5), we conclude that 
\begin{align}
    \mathcal{S}_1 = 0.
\end{align}
Looking at the inner sum of $\mathcal{S}_2$, we observe that
\begin{align}
    \sum_{J' \supset I_{j+1}...I_s} \mu_K(J')\omega_K(J') = \sum_{i=1}^{t}(-1)^i{t \choose i}i.
\end{align}
On the other hand, if we take the derivative on both sides of (2.4), we get that
\begin{align}
    -t(1-x)^{t-1} = \sum_{i=1}^{t}(-1)^i{t \choose i}ix^{i-1}.
\end{align}
Hence, again putting $x=1$ on the RHS of (2.8) yields the RHS of (2.7), and therefore, from equating the left hand sides of (2.7) and (2.8) for $x=1$, we get that 
\begin{align}
      \sum_{J' \supset I_{j+1}...I_s} \mu_K(J')\omega_K(J') = \begin{cases}
          -1,\;\text{when $I_{j+1}...I_s$ is a prime ideal or a prime power} \\
          0,\;\text{otherwise.}
      \end{cases}
\end{align}
Therefore, the only contribution from $\mathcal{S}_2$ is when $j=s-1$ and that $I_s$ is either just a prime ideal, or its power. Note that by our construction, $I_s$ is the product of the prime ideals containing the ideal $I$ with the largest norm and hence, by our hypothesis, $I_s$ is indeed a prime ideal or a prime power. Therefore, we have from (2.2) and (2.9) that
\begin{align}
    \mathcal{S}_2 = -\sum_{\substack{J \supset I_{s-1} \\ J \not \supset I_s \\ J \neq \mathcal{O}_K}} \mu_K(J)f(J) = Q_C^2(I).
\end{align}
To understand the last equality, we observe that $I_{s-1}$ is a product of prime ideals containing $I$ with the second largest norm. Now, the only positive contribution from the terms in the sum in (2.10) is when we choose $J$ to just be the prime ideals, since:
\begin{itemize}
    \item[(i)] we cannot choose $J$ contained in powers of prime ideals as then $\mu_K(J)=0$, and
    \item[(ii)] we cannot choose $J$ contained in two distinct prime ideals, as then $f(J)=0$\footnote[6]{note that every such $J$ is contained in multiple prime ideals of the smallest norm}. 
\end{itemize}
Thus, along with a negative sign, the sum counts the number of prime ideals containing the ideal $I$ with the second largest norm. Hence, combining (2.2), (2.6), and (2.10), we conclude our theorem.
\end{proof}
\noindent Therefore, we have what we call the second order duality between prime ideals in a number ring. We note that the RHS of \textit{Lemma 2.2} counts the number of prime ideals with the largest norm and that of \textit{Theorem 2.3} counts the number of prime ideals with the second largest norm. Hence, we call them first and second order dualities, respectively.  Let us now state and prove the general higher order dualities between prime ideals. The proof will follow an idea similar to that used in the proof of \textit{Theorem 2.3}.
\begin{theorem}
    Let $K$ be a number field and $k$ be a positive integer. Let $I \subset \mathcal{O}_K$ be an ideal such that $Q^i(I)=1$, for $1 \leq i\leq k-1$. Let $f$ be the characteristic function of the set $S(L/K;C)$. Then the following identity holds:
    \begin{align}
        \sum_{J \supset I}\mu_K(J){{\omega_K(J)-1} \choose k-1}f(J) = (-1)^kQ_C^k(I) .\nonumber
    \end{align}
\end{theorem}
\begin{proof}
    A priori, we must note that this theorem only makes sense when the chosen ideal $I$ has prime factors with at least $k$ distinct norms.
    We now start with our usual factorization of $I$ as done above: $I = I_1...I_s$ where the $I_i$'s are products of prime ideals, containing $I$, of the same norm. Without the loss of generality, let us assume that the factors are written in an ascending order of their norms, i.e. $N(I_i) < N(I_j)$, for all $i < j$. From the LHS of the above equation, we have
    \begin{align}
         &\sum_{J \supset I}\mu_K(J){{\omega_K(J)-1} \choose k-1}f(J) \nonumber \\ =& \frac{1}{(k-1)!} \sum_{J \supset I}\mu_K(J)(\omega_K(J)-1)...(\omega_K(J)-k+1)f(J) \nonumber \\
         =&\frac{1}{(k-1)! } \sum_{j=1}^{s-k+1} \sum_{\substack{J \supset I_j \\ J \not \supset I_{k'} \forall k' \geq j+1 \\ J \neq \mathcal{O}_K}} \mu_K(J)f(J)\sum_{J'\supset I_{j+1}...I_s} \mu_K(J')(\omega_K(J)+\omega_K(J')-1)...(\omega_K(J)+\omega_K(J')-k+1).
    \end{align}
    We note that the summand of the innermost sum in (2.11)  can be represented as 
    \begin{align}
        \sum_{u=0}^{k-1}\mu_K(J')c_u(\omega_K(J))\omega_K(J')^u.
    \end{align}
where $c_u(\omega_K(J))$ are coefficients of the polynomial in $\omega_K(J')$. Since the sums are finite, we can do an interchange of the summations to get
\begin{align}
    \sum_{J'\supset I_{j+1}...I_s} \mu_K(J')(\omega_K(J)+\omega_K(J')-1)...(\omega_K(J)+\omega_K(J')-k+1) = \sum_{u=0}^{k-1}c_u(\omega_K(J))\sum_{J'\supset I_{j+1}...I_s}\mu_K(J')\omega_K(J')^u .
\end{align}
Now we note that $j$ ranges from $1$ to $s-k+1$. For $0 \leq u \leq k-1$, we can write the inner sum in the RHS of (2.13) as 
\begin{align}
    \sum_{J'\supset I_{j+1}...I_s}\mu_K(J')\omega_K(J')^u = \sum_{i=1}^{t}(-1)^i{t \choose i}i^u .
\end{align}
On the other hand, looking at the binomial expansion
\begin{align}
    (1-x)^t = \sum_{i=0}^t(-1)^i{t\choose i}x^i,
\end{align}
we take derivative with respect to $x$ and multiply with $x$ on both sides of (2.15) to get
\begin{align}
    -tx(1-x)^{t-1} = \sum_{i=1}^{t}(-1)^i{t \choose i}ix^i.
\end{align}
Repeating the same step again, we see 
\begin{align}
    -tx(1-x)^{t-1}+t(t-1)x^2(1-x)^{t-2} = \sum_{i=1}^{t}(-1)^i{t \choose i}i^2x^i.
\end{align}
We continue the same for $u$ times and get
\begin{align}
    \sum_{\ell=1}^{u}\Tilde{c}_{\ell}(t)x^{\ell}(1-x)^{t-\ell} = \sum_{i=1}^{t} (-1)^i{t \choose i}i^ux^i.
\end{align}
With $x=1$, the RHS of (2.18) is exactly the RHS of (2.14), while the LHS of (2.18) with $x=1$ is non-zero only when $\ell = t$. That is only possible if $t\leq u$. Now, $t$ is the number of distinct prime ideals containing the product $I_{j+1}...I_s$, where $j \leq s-k+1$. Hence, by our hypothesis, $k-1 \leq t$. But $u \leq k-1$, and therefore, the only non-zero contribution occurs on the RHS of (2.13) when $u=k-1$. In other words, we can also say that for $0 \leq u \leq k-2$, we have that 
\begin{align}
    \sum_{J'\supset I_{j+1}...I_s}\mu_K(J')\omega_K(J')^u  = 0 .
\end{align}
Thus, from (2.11), (2.13), and (2.19), we continue to get
\begin{align}
     &\sum_{J \supset I}\mu_K(J){{\omega_K(J)-1} \choose k-1}f(J) \nonumber \\ =& \frac{1}{(k-1)! } \sum_{j=1}^{s-k+1} \sum_{\substack{J \supset I_j \\ J \not \supset I_{k'} \forall k' \geq j+1 \\ J \neq \mathcal{O}_K}} \mu_K(J)f(J)  \sum_{J'\supset I_{j+1}...I_s}\mu_K(J')\omega_K(J')^{k-1}.
\end{align}
Using (2.19), we can again rewrite (2.20) as
\begin{align}
     &\sum_{J \supset I}\mu_K(J){{\omega_K(J)-1} \choose k-1}f(J) \nonumber \\ =& \frac{1}{(k-1)! } \sum_{j=1}^{s-k+1} \sum_{\substack{J \supset I_j \\ J \not \supset I_{k'} \forall k' \geq j+1 \\ J \neq \mathcal{O}_K}} \mu_K(J)f(J)  \sum_{J'\supset I_{j+1}...I_s}\mu_K(J') \omega_K(J')(\omega_K(J')-1)...(\omega_K(J')-k+2).
\end{align}
Again, the inner sum of (2.21) can be written as
\begin{align}
    & \sum_{J'\supset I_{j+1}...I_s}\mu_K(J') \omega_K(J')(\omega_K(J')-1)...(\omega_K(J')-k+2) \nonumber \\
    =& \sum_{i=k-1}^t (-1)^i {t \choose i}i(i-1)...(i-k+2) = \left(\frac{d^k}{dx^k}\left[\sum_{i=0}^t(-1)^i{t \choose i}x^i\right]\right)\Bigg|_{x=1} = \left(\frac{d^k}{dx^k}(1-x)^t\right)\Bigg|_{x=1} \nonumber \\
    =& \begin{cases}
        (-1)^{k-1}(k-1)!,\;\text{when $I_{j+1}...I_s$ has exactly $k-1$ distinct prime factors} \\
        0,\;\text{otherwise.}
    \end{cases}
\end{align}
Therefore, by our hypothesis, for the LHS of (2.22) to give a non-zero contribution, $j$ can only take the value of $s-k+1$. Therefore, from (2.21) and (2.22), we get 
\begin{align}
     \sum_{J \supset I}\mu_K(J){{\omega_K(J)-1} \choose k-1}f(J) = \frac{(-1)^{k-1}(k-1)!}{(k-1)!} \sum_{\substack{J \supset I_{s-k+1} \\ J \not \supset I_k' \forall k' \geq s-k\\ J \neq \mathcal{O}_K }} \mu_K(J)f(J) = (-1)^kQ_C^k(I).
\end{align}
This proves our general duality lemma for any positive integer $k$ and for ideals satisfying some conditions in any arbitrary number ring $K$. 
\end{proof}

\begin{remark}
    We observe that plugging $k=1$ and $k=2$ in \textit{Theorem 2.4} gives us \textit{Lemma 2.2} and \textit{Theorem 2.3} respectively. Therefore, they both are consistent with our general duality identity.
\end{remark}

\noindent In the following section, we will prove estimates of functions counting prime ideals that will be essential for us in proving our main results that will follow later, where the above duality identities will be used predominantly.

\section{The Prime Ideal Counting Functions and their Estimates}
Let us recall the Dickman function [De61] $\rho(\beta)$ defined to be the continuous solution to the following system of equations:
\begin{align}
    \rho(\beta) &= 1,\;\text{for $0 \leq \beta \leq 1$}, \nonumber \\
    -\beta\rho'(\beta) &= \rho(\beta -1),\;\text{for $\beta>1$} .\nonumber
\end{align}
de Bruijn [Br51] and Hildebrand [Hi86] used the Dickman function to estimate a well-known prime counting function defined as
\begin{align}
    \Psi(x,y) = \sum_{\substack{x \leq n \\ P_1(n)\leq y}}1.
\end{align}
Tenenbaum [Te00] has also proved several estimates involving the $\Psi(x,y)$ function and the $k^{th}$ largest prime factors, which proved to be very useful in proving results in the classical case. But here, as we shift from the classical to an algebraic setting, we need to consider the algebraic analogues of such functions and their corresponding estimates. We count, what are known as \textit{smooth ideals}, using the following function: for a number field $K$, 
\begin{align}
    \Psi_K(X,Y) = \sum_{\substack{N(I) \leq X \\ M_1(I) \leq Y}} 1.
\end{align}
counts the number of ideals in $\mathcal{O}_K$ with their norm bounded by $X$ such that the norm of their prime factors is bounded by $Y$. It has been proved by Krause [Kr90] and Moree [Mo92] that an asymptotic estimate for the function $\Psi_K$ exist and is given by
\begin{align}
    \Psi_K(X,Y) =X\rho(\beta) \left(1+O_{\varepsilon}\left(\frac{\beta\log(\beta+1)}{\log X}\right)\right),
\end{align}
where $\beta = \frac{\log X}{\log Y}$. Note that this asymptote is uniform for $1 \leq \beta \leq (\log X)^{1-\varepsilon}$. (3.3) also corresponds to Hildebrand's result [Hi86] in the classical case, i.e. when $K=\mathbb{Q}$. The implicit constant in the error of (3.3) depends on $\varepsilon$, which can be suitably chosen to derive desired estimates. As done in the classical case by Maier [Ma(up)] as a corollary to Hildebrand's result, in the algebraic setting too, one can use the following upper bound of the Dickman function (see Norton [No71])
\begin{align}
    \rho(\beta) \leq \frac{1}{\Gamma(\beta+1)} \nonumber
\end{align}
to derive the asymptotic
\begin{align}
    \rho(\beta) \sim \frac{1}{\sqrt{2\pi\beta}}e^{-\beta\log\left(\frac{\beta}{e}\right)},
\end{align}
and then, use (3.4) to get another useful estimate of $\Psi_K(X,Y)$, i.e.
\begin{align}
    \Psi_K(X,Y) \ll Xe^{-c_3\beta},
\end{align}
for some constant $c_3$. We will use both the estimates (3.3) and (3.5) in different contexts in our proofs. As mentioned earlier, our aim is to derive new formulas of the Chebotarev Density using higher order dualities between prime ideals as proved above. In this paper, we will mainly focus on using the second order duality\footnote[7]{can be viewed as generalizations of results by Alladi-Johnson [AJ24] and the author [Se25]} and comment on how the higher order dualities can be utilized following in the lines of the proofs by Alladi-Sengupta [AS(ip)]. In what follows, we will prove a few important lemmas in this section that we will use to prove our main results. \\ \\
Looking at the hypothesis of \textit{Theorem 2.3}, we observe that it is required to consider only such ideals that satisfy the condition $Q^1(I)=1$. We now prove that the number of ideals that do not satisfy such a condition is small compared to a bound on the norm of the ideals. In fact, in the lemma that follows, we prove something even stronger.
\begin{lemma}
    Let $\mathcal{S}_1(K;X)$ be a set of ideals in $\mathcal{O}_K$ defined as
    \begin{align}
        \mathcal{S}_1(K;X): = \{I \subset \mathcal{O}_K:\;M_1(I)^2|N(I);\;N(I)\leq X\}.
    \end{align}
    Denoting $\mathcal{N}_1(K;X) = \#\mathcal{S}_1(K;X)$, we have
    \begin{align}
        \mathcal{N}_1(K;X) \ll \frac{X}{e^{c_3\sqrt{\log X\log\log X}}}.
    \end{align}
\end{lemma}
\begin{proof}
    Using the estimate of $\Psi_K(X,Y)$ in (3.5), with the choice of $Y = e^{\sqrt{\frac{\log X}{\log\log X}}}$, we have that 
    \begin{align}
        \Psi_K\left(X,e^{\sqrt{\frac{\log X}{\log\log X}}}\right) \ll \frac{X}{e^{c_3\sqrt{\log X\log \log X}}}.
    \end{align}
    Now, it only suffices to look at those ideals satisfying $M_1(I)=N(\mathfrak{p})>e^{\sqrt{\frac{\log X}{\log\log X}}}$. The number of such ideals is essentially of the order of the number of ideals with norm less than or equal to $\frac{X}{N(\mathfrak{p})^2}$. Therefore, by \textit{Theorem 1.5}, we have that 
    \begin{align}
        N_1(K;X) &\ll  \Psi_K\left(X,e^{\sqrt{\frac{\log X}{\log\log X}}}\right) + c_K \sum_{N(\mathfrak{p})>e^{\sqrt{\frac{\log X}{\log\log X}}}}\frac{X}{N(\mathfrak{p})^2} \nonumber \\
        &\ll \frac{X}{e^{c_3\sqrt{\log X\log\log X}}} + c_K\sum_{n>e^{\sqrt{\frac{\log X}{\log\log X}}}} \frac{X}{n^2} \nonumber \\
        &\ll \frac{X}{e^{c_3\sqrt{\log X\log\log X}}}.
    \end{align}
This proves our required lemma.
\end{proof}
\noindent We now define another ideal counting function: the function
\begin{align}
    \Psi_{K,2}(X,Y) := \sum_{\substack{N(I)\leq X \\ M_2(I)\leq Y}} 1 
\end{align}
counts the number of ideals in $\mathcal{O}_K$ with norm bounded by $X$, and the second largest norm among all their prime ideal factors is bounded by $Y$. The next lemma will provide us with a crucial estimate of $\Psi_{K,2}(X,Y)$ with $Y$ belonging to some given interval. This estimate is for the number of ideals with all the prime factors, except those with the largest norm, of smaller norms. From the estimate (3.5), it is quite clear that $\Psi_K(X,Y)$ is quite small as compared to $X$, for large $X$. On the other hand, $\Psi_{K,2}(X,Y)$ is not that small. This can be viewed by considering ideals of the form $I=\mathfrak{p}_1\mathfrak{p}_2$, such that $N(\mathfrak{p}_1)=2<N(\mathfrak{p}_2)$. We observe that 
\begin{align}
    \Psi_{K,2}(X,Y) \geq \Psi_{K,2}(X,2), 
\end{align}
and $\Psi_{K,2}(X,2)$ is more than the number of prime ideals with norm no more than $\frac{X}{2}$. Hence, continuing from (3.11), using the asymptotic version of the Prime Ideal Theorem, we have
\begin{align}
     \Psi_{K,2}(X,Y) \gg \frac{X}{\log \frac{X}{2}} .\nonumber
\end{align}
Therefore, it is important to find its estimate, of course, under some restrictions, which is provided by the following lemma.
\begin{lemma}
    For $Y \leq e^{(\log X)^{1-\delta}}$, for some small $\delta >0$, we have
    \begin{align}
        \Psi_{K,2}(X,Y) \ll \frac{X\log Y}{\log X}.
    \end{align}
\end{lemma}
\begin{proof}
    Let $\mathfrak{p} \subset \mathcal{O}_K$ be a prime ideal. We define a new counting function as follows:
    \begin{align}
        \Psi_{K,2}(X,\mathfrak{p}) = \sum_{\substack{N(I) \leq X \\ M_2(I)=N(\mathfrak{p})}}1.
    \end{align}
    Therefore, from (3.10) and (3.13), we see that the following inequality holds:
    \begin{align}
        \Psi_{K,2}(X,Y) \leq \sum_{\substack{\mathfrak{p} \subset \mathcal{O}_K \\ N(\mathfrak{p}) \leq Y}}\Psi_{K,2}(X,\mathfrak{p}).
    \end{align}
    The reason behind the inequality above is that in the sum on the left, we are counting ideals, whereas on the right, we are counting prime ideals with bounded norm, and therefore, there might be over-counting due to the presence of an ideal inside multiple prime ideals of the same norm. Therefore, estimating the sum on the right of (3.14) would suffice for our required result. \\ \\
    Let us first fix a prime ideal $\mathfrak{p}$ and consider an ideal $I \subset \mathcal{O}_K$ such that it can be represented as $I = \mathfrak{m.pq}$, where $M_1(I)=\mathfrak{q}$, $M_2(I)=\mathfrak{p}$ and $M_1(\mathfrak{m}) \leq N(\mathfrak{p})$. Note, we are considering ideals such that $Q^1(I)=1$, where we denote the unique prime ideal with the largest norm as $\mathfrak{q}$, such that $I \not \subset \mathfrak{q}^2$, as by \textit{Lemma 3.1}, the number of the ideals left out is small and will only contribute to the error term. Rewriting $\Psi_{K,2}(X,\mathfrak{p})$ using its definition (3.13), we have
    \begin{align}
          \Psi_{K,2}(X,\mathfrak{p}) = \sum_{\substack{N(\mathfrak{m}) \leq \frac{X}{N(\mathfrak{p})^2} \\ M_1(\mathfrak{m})\leq N(\mathfrak{p})}}\sum_{N(\mathfrak{p})\leq N(\mathfrak{q})\leq \frac{X}{N(\mathfrak{mp})}} 1 + O\left(\frac{X}{e^{c_3\sqrt{\log X\log\log X}}}\right).
    \end{align}
    We use the asymptotic version of the Prime Ideal Theorem to get from (3.15) that
    \begin{align}
          \Psi_{K,2}(X,\mathfrak{p}) \ll \sum_{\substack{N(\mathfrak{m}) \leq \frac{X}{N(\mathfrak{p})^2} \\ M_1(\mathfrak{m})\leq N(\mathfrak{p})}} \frac{X}{N(\mathfrak{mp})\log \left(\frac{X}{N(\mathfrak{mp})}\right)}.
    \end{align}
    Now, let us denote $T=T(X)=e^{(\log X)^{1-\frac{\delta}{2}}}$. Truncating the sum on the RHS of (3.16) at T, we get that
\begin{align}
    \Psi_{K,2}(X,\mathfrak{p}) \ll \sum_{\substack{N(\mathfrak{m}) \leq T \\ M_1(\mathfrak{m})\leq N(\mathfrak{p})}} \frac{X}{N(\mathfrak{mp})\log \left(\frac{X}{N(\mathfrak{mp})}\right)} +  \sum_{\substack{T<N(\mathfrak{m}) \leq \frac{X}{N(\mathfrak{p})^2} \\ M_1(\mathfrak{m})\leq N(\mathfrak{p})}} \frac{X}{N(\mathfrak{mp})\log \left(\frac{X}{N(\mathfrak{mp})}\right)} .
\end{align}
To estimate the first sum on the right, we first note that $\log\left(\frac{X}{N(\mathfrak{mp})}\right) \gg \log X$, since $N(\mathfrak{m})\leq T$ and $N(\mathfrak{p})\leq Y$. Therefore, we have 
\begin{align}
    \sum_{\substack{N(\mathfrak{m}) \leq T \\ M_1(\mathfrak{m})\leq N(\mathfrak{p})}} \frac{X}{N(\mathfrak{mp})\log \left(\frac{X}{N(\mathfrak{mp})}\right)} \ll \frac{X}{N(\mathfrak{p})\log X} \sum_{\substack{N(\mathfrak{m}) \leq T \\ M_1(\mathfrak{m})\leq N(\mathfrak{p})}} \frac{1}{N(\mathfrak{m})} \ll \frac{X}{N(\mathfrak{p})\log X} \prod_{\substack{\mathfrak{t} \subset \mathcal{O}_K\\ \mathfrak{t}\text{ is prime} \\ N(\mathfrak{t}) \leq N(\mathfrak{p})}}\left(1-\frac{1}{\mathfrak{t}}\right)^{-1}.
\end{align}
The inner product on the right can then be estimated using Merten's Theorem for prime ideals [Le23] as follows:
\begin{align}
    \prod_{\substack{\mathfrak{t} \subset \mathcal{O}_K\\ \mathfrak{t}\text{ is prime} \\ N(\mathfrak{t}) \leq N(\mathfrak{p})}}\left(1-\frac{1}{\mathfrak{t}}\right)^{-1}  \ll \log N(\mathfrak{p}).
\end{align}
Hence, from (3.18) and (3.19), we continue to get
\begin{align}
    \sum_{\substack{N(\mathfrak{m}) \leq T \\ M_1(\mathfrak{m})\leq N(\mathfrak{p})}} \frac{X}{N(\mathfrak{mp})\log \left(\frac{X}{N(\mathfrak{mp})}\right)} \ll \frac{X\log N(\mathfrak{p})}{N(\mathfrak{p})\log X}.
\end{align}
To estimate the second sum in the RHS of (3.17), we again observe that $\log N(\mathfrak{p}) \leq \log\left(\frac{X}{N(\mathfrak{mp})}\right)$, since $N(\mathfrak{m}\mathfrak{p}^2) \leq X$. Thus, using this, we 
\begin{align}
    \sum_{\substack{T<N(\mathfrak{m}) \leq \frac{X}{N(\mathfrak{p})^2} \\ M_1(\mathfrak{m})\leq N(\mathfrak{p})}} \frac{X}{N(\mathfrak{mp})\log \left(\frac{X}{N(\mathfrak{mp})}\right)} \ll \frac{X}{N(\mathfrak{p})\log N(\mathfrak{p})}  \sum_{\substack{T<N(\mathfrak{m}) \leq \frac{X}{N(\mathfrak{p})^2} \\ M_1(\mathfrak{m})\leq N(\mathfrak{p})}} \frac{1}{N(\mathfrak{m})}.
\end{align}
We estimate the inner sum on the right of (3.21) and in pursuit of that, we denote $N(\mathfrak{m})=n$. Then rewriting from (3.21), we get 
\begin{align}
     \sum_{\substack{T<N(\mathfrak{m}) \leq \frac{X}{N(\mathfrak{p})^2} \\ M_1(\mathfrak{m})\leq N(\mathfrak{p})}} \frac{1}{N(\mathfrak{m})} = \sum_{T<n\leq \frac{X}{N(p)^2}} \frac{1}{n} \sum_{\substack{\mathfrak{m} \subset \mathcal{O}_K \\ N(\mathfrak{m})=n \\ M_1(\mathfrak{m}) \leq N(\mathfrak{p})}}1 \leq \sum_{T<n \leq \frac{X}{N(\mathfrak{p})^2}}\frac{\Psi_K(n,N(\mathfrak{p}))}{n} \ll \sum_{T<n\leq \frac{X}{N(\mathfrak{p})^2}} e^{-c_3\alpha} ,
\end{align}
where $\alpha =\frac{\log n}{\log N(\mathfrak{p})}$. Now, we take into consideration the constraint that we have, i.e. $N(\mathfrak{p}) \leq Y$, and by hypothesis, $Y \leq e^{(\log X)^{1-\delta}}$. Therefore, we can write
\begin{align}
    \alpha > \frac{\log T}{\log Y} = \frac{\log e^{(\log X)^{1-\frac{\delta}{2}}}}{\log e^{(\log X)^{1-\delta}}} = (\log X)^{\frac{\delta}{2}}.
\end{align}
Thus, from (3.22) and (3.23), we get that
\begin{align}
    \sum_{\substack{T<N(\mathfrak{m}) \leq \frac{X}{N(\mathfrak{p})^2} \\ M_1(\mathfrak{m})\leq N(\mathfrak{p})}} \frac{1}{N(\mathfrak{m})} \ll e^{-c_3(\log X)^{\frac{\delta}{2}}} .
\end{align}
Therefore, combining (3.17), (3.20), (3.21), and (3.24), we finally have that 
\begin{align}
     \Psi_{K,2}(X,\mathfrak{p}) \ll \frac{X\log N(\mathfrak{p})}{N(\mathfrak{p})\log X}+\frac{X}{N(\mathfrak{p})\log N(\mathfrak{p})}e^{-c_3(\log X)^{\frac{\delta }{2}}}.
\end{align}
Therefore, summing all the prime ideals $\mathfrak{p} \subset \mathcal{O}_K$ with norms bounded by $Y$, we have from (3.25)
\begin{align}
    \sum_{\substack{\mathfrak{p} \subset \mathcal{O}_K \\ N(\mathfrak{p}) \leq Y}}\Psi_{K,2}(X,\mathfrak{p}) &\ll  \sum_{\substack{\mathfrak{p} \subset \mathcal{O}_K \\ N(\mathfrak{p}) \leq Y}}\frac{X\log N(\mathfrak{p})}{N(\mathfrak{p})\log X} +  \sum_{\substack{\mathfrak{p} \subset \mathcal{O}_K \\ N(\mathfrak{p}) \leq Y}} \frac{X}{N(\mathfrak{p})\log N(\mathfrak{p})}e^{-c_3(\log X)^{\frac{\delta }{2}}} \nonumber \\ 
   &\ll \frac{X}{\log X}\sum_{\substack{\mathfrak{p} \subset \mathcal{O}_K \\ N(\mathfrak{p}) \leq Y}} \frac{\log N(\mathfrak{p})}{N(\mathfrak{p})}+ \frac{X}{e^{(\log X)^{\frac{\delta}{2}}}} \sum_{\substack{\mathfrak{p} \subset \mathcal{O}_K \\ N(\mathfrak{p}) \leq Y}} \frac{1}{N(\mathfrak{p})\log N(\mathfrak{p})} \nonumber \\
   &\ll \frac{X\log Y}{\log X}.
\end{align}
The last estimate uses a well-known number field version of Merten's Theorem [Le23] in the first sum, while the second sum is convergent and is much smaller in comparison to the first sum. Thus, we prove our lemma. 
\end{proof}
\begin{remark}
    Although we have the estimate for $\Psi_{K,2}(X,Y)$, we will be using the estimate in (3.26), i.e.
    \begin{align}
         \sum_{\substack{\mathfrak{p} \subset \mathcal{O}_K \\ N(\mathfrak{p}) \leq Y}}\Psi_{K,2}(X,\mathfrak{p})
\ll \frac{X\log Y}{\log X}.
    \end{align}
    for our proofs.
\end{remark}
\noindent Therefore, \textit{Lemma 3.2} gives us the useful estimate (3.27) that will come in handy to prove the important results in the next section. Before that, we will need another important lemma that gives us a treatment of the number of ideals in $\mathcal{O}_K$ contained in at least two prime ideals with the second largest norm, i.e., an estimate on the number of ideals satisfying $Q^2(I)\geq 2$. It will be more evident as to why we need this lemma, once we use it in the next section (see equation (4.21) in the proof of \textit{Theorem 4.1}).
\begin{lemma}
    Given $X \geq 2$ and $Y \leq e^{(\log X)^{1-\delta}}$, for some $\delta >0$, we have for some positive constant $c_1$ that
    \begin{align}
        \sum_{\substack{I \subset \mathcal{O}_K \\ 2 \leq N(I) \leq X \\ Q^2(I) \geq 2}}(Q^2(I)-1) \ll \frac{X\log Y}{\log X}+\frac{X\log\log X}{e^{c_1\sqrt{\log Y}}}. \nonumber
    \end{align}
\end{lemma}
\begin{proof}
    In this proof, we will use the estimate (3.3) of $\Psi_K(X,Y)$. We start by defining the set $S_2\left(\frac{X}{N(\mathfrak{p})},\mathfrak{p}\right)$ of ideals $I \subset \mathcal{O}_K$ such that $N(I) \leq \frac{X}{N(\mathfrak{p})}$ with a unique and non-repeating prime ideal of the largest norm and $M_2(I)=N(\mathfrak{p})$. Therefore, by the definition of $Q^2(I)$ and using \textit{Lemma 3.1}, we get
    \begin{align}
        \sum_{\substack{I \subset \mathcal{O}_K \\ 2 \leq N(I) \leq X \\ Q^2(I) \geq 2}}(Q^2(I)-1) \leq \sum_{N(\mathfrak{p}) \leq \sqrt{X}} \left|S_2\left(\frac{X}{N(\mathfrak{p})},\mathfrak{p}\right)\right| +  O\left(\frac{X}{e^{c_3\sqrt{\log X\log\log X}}}\right).
    \end{align}
    Note $|\cdot|$ above determines the cardinality function. Therefore, to get our desired result, we need to estimate $\left|S_2\left(\frac{X}{N(\mathfrak{p})},\mathfrak{p}\right)\right|$. To do that, we choose $I \in S_2\left(\frac{X}{N(\mathfrak{p})},\mathfrak{p}\right)$. Then, we can write $I = \mathfrak{m.pq}$ such that $M_1(I)=N(\mathfrak{q})$, $M_2(I)=N(\mathfrak{p})$ and $\mathfrak{m} \subset \mathcal{O}_K$ is an ideal such that $M_1(\mathfrak{m}) \leq N(\mathfrak{p})$. Therefore, using the definition of $\Psi_K(X,Y)$, we have
    \begin{align}
        \left|S_2\left(\frac{X}{N(\mathfrak{p})},\mathfrak{p}\right)\right| &= \sum_{\substack{N(\mathfrak{m}) \leq \frac{X}{N(\mathfrak{p})^3} \\ M_1(\mathfrak{m})\leq N(\mathfrak{p})}}\sum_{\substack{N(\mathfrak{p})<N(\mathfrak{q}) \\ N(\mathfrak{m.pq})\leq \frac{X}{N(\mathfrak{p})}}}1 = \sum_{N(\mathfrak{p})<N(\mathfrak{q})\leq \frac{X}{N(p)^2}}\sum_{\substack{N(\mathfrak{m})\leq \frac{X}{N(\mathfrak{p})^2N(\mathfrak{q})} \\ M_1(\mathfrak{m})\leq N(\mathfrak{p})}} 1 \nonumber \\
    &= \sum_{\substack{\mathfrak{q} \in \mathcal{O}_K \\ \mathfrak{q} \text{ is a prime} \\N(\mathfrak{p})<N(\mathfrak{q})\leq \frac{X}{N(\mathfrak{p})^2}}} \Psi_K\left(\frac{X}{N(\mathfrak{p})^2N(\mathfrak{q})},N(\mathfrak{p})\right).
    \end{align}
    Using (3.27), we have
\begin{align}
    \sum_{N(\mathfrak{p})\leq Y} \left|S_2\left(\frac{X}{N(\mathfrak{p})},\mathfrak{p}\right)\right|\leq   \sum_{N(\mathfrak{p})\leq Y} \Psi_{K,2}\left(\frac{X}{N(\mathfrak{p})},\mathfrak{p}\right) \ll \frac{X\log Y}{\log X}.
\end{align}
Hence, in view of (3.28), we are only required to prove the estimate of the remaining interval of $N(\mathfrak{p})$, i.e.
\begin{align}
    \sum_{Y<N(\mathfrak{p}) \leq\sqrt{X}} \left|S_2\left(\frac{X}{N(\mathfrak{p})},\mathfrak{p}\right)\right| = \sum_{Y < N(\mathfrak{p})\leq \sqrt{X}} \;\sum_{N(\mathfrak{p})<N(\mathfrak{q})\leq \frac{X}{N(\mathfrak{p})^2}} \Psi_K\left(\frac{X}{N(\mathfrak{p})^2N(\mathfrak{q})},N(\mathfrak{p})\right).
\end{align}
We implement a change in the order of summation in the RHS of (3.31), which then involves a split in the double sum. We observe that $N(\mathfrak{p})$ satisfies the following inequalities:
\begin{align}
    Y<N(\mathfrak{p}), \quad N(\mathfrak{p})<N(\mathfrak{q}), \quad N(\mathfrak{p})\leq \frac{X}{N(\mathfrak{p})N(\mathfrak{q})},\quad N(\mathfrak{p}) \leq \sqrt{X} .\nonumber
\end{align}
Therefore, with the change in the order of summations and the split, (3.31) results in
\begin{align}
    \sum_{\substack{Y<N(\mathfrak{p}) \leq \sqrt{X} }} \left|S_2\left(\frac{X}{N(\mathfrak{p})},\mathfrak{p}\right)\right| &= \sum_{\substack{Y <N(\mathfrak{q}) \leq \sqrt{X}}} \;  \sum_{\substack{Y<N(\mathfrak{p})<N(\mathfrak{q}) }} \Psi_K\left(\frac{X}{N(\mathfrak{p})^2N(\mathfrak{q})},N(\mathfrak{p})\right) \nonumber \\ &\hspace{2cm}+ \sum_{\sqrt{X} < N(\mathfrak{q}) \leq \frac{X}{Y^2}} \;\sum_{\substack{Y<N(\mathfrak{p})\leq \sqrt{\frac{X}{N(\mathfrak{q})} } }} \Psi_K\left(\frac{X}{N(\mathfrak{p})^2N(\mathfrak{q})},N(\mathfrak{p})\right) .
\end{align}
We will now use a trick that will not only prove to be very useful here in this proof, but also in the proof of \textit{Theorem 4.1}. This step enables us to use the strong form of the Prime Ideal Theorem (\textit{Theorem 1.3}) to obtain our estimates\footnote[8]{later we will observe that this trick also enables us to involve the strong form of CDT (see \textit{Theorem 4.1})}. Here, we replace the inner sums of both the double sums in the RHS of (3.32) using the following integrals:
\begin{align}
     \sum_{\substack{Y<N(\mathfrak{p})<N(\mathfrak{q}) }} \Psi_K\left(\frac{X}{N(\mathfrak{p})^2N(\mathfrak{q})},N(\mathfrak{p})\right) \quad &\text{by} \quad \int_Y^{N(\mathfrak{q})} \Psi_K\left(\frac{X}{t^2N(\mathfrak{q})},t\right)\frac{dt}{\log t} ,\nonumber \\
     \sum_{\substack{Y<N(\mathfrak{p})<\sqrt{\frac{X}{N(\mathfrak{q})} } }} \Psi_K\left(\frac{X}{N(\mathfrak{p})^2N(\mathfrak{q})},N(\mathfrak{p})\right) \quad &\text{by} \quad \int_Y^{\sqrt{\frac{X}{N(\mathfrak{q})} } } \Psi_K\left(\frac{X}{t^2N(\mathfrak{q})},t\right)\frac{dt}{\log t}. \nonumber
\end{align}
Indeed, such a change will lead to some error terms originating from the differences between the existing and the newly introduced expressions, namely
\begin{align}
    E&: = \sum_{\substack{Y <N(\mathfrak{q}) \leq \sqrt{X}}}\left(   \sum_{\substack{Y<N(\mathfrak{p})<N(\mathfrak{q}) }} \Psi_K \left(\frac{X}{N(\mathfrak{p})^2N(\mathfrak{q})},N(\mathfrak{p})\right) - \int_Y^{N(\mathfrak{q})} \Psi_K\left(\frac{X}{t^2N(\mathfrak{q})},t\right)\frac{dt}{\log t}          \right), \nonumber \\
    \Tilde{E} &:= \sum_{\sqrt{X} < N(\mathfrak{q}) \leq \frac{X}{Y^2}} \left(        \sum_{\substack{Y<N(\mathfrak{p})<\sqrt{\frac{X}{N(\mathfrak{q})} } }} \Psi_K\left(\frac{X}{N(\mathfrak{p})^2N(\mathfrak{q})},N(\mathfrak{p})\right) - \int_Y^{\sqrt{\frac{X}{N(\mathfrak{q})} } } \Psi_K\left(\frac{X}{t^2N(\mathfrak{q})},t\right)\frac{dt}{\log t}           \right) .\nonumber
\end{align}
Estimating $E$, we make some change in the order of summations to derive
\begin{align}
    |E| &= \sum_{\substack{Y <N(\mathfrak{q}) \leq \sqrt{X}}}\left|   \sum_{\substack{Y<N(\mathfrak{p})<N(\mathfrak{q}) }} \Psi_K\left(\frac{X}{N(\mathfrak{p})^2N(\mathfrak{q})},N(\mathfrak{p})\right) - \int_Y^{N(\mathfrak{q})} \Psi_K\left(\frac{X}{t^2N(\mathfrak{q})},t\right)\frac{dt}{\log t}          \right| \nonumber \\
    &=  \sum_{Y<N(\mathfrak{q})\leq \sqrt{X}}\left|\sum_{\substack{Y<N(\mathfrak{p})<N(\mathfrak{q}) }}\; \sum_{\substack{N(\mathfrak{m}) \leq \frac{X}{N(\mathfrak{p})^2N(\mathfrak{q})} \\ M_1(\mathfrak{m})\leq N(\mathfrak{p})}} 1- \int_Y^{N(\mathfrak{q})} \left\{\sum_{\substack{N(\mathfrak{m}) \leq \frac{X}{t^2N(\mathfrak{q})} \\ M_1(\mathfrak{m})\leq t}}1\right\}\frac{dt}{\log t}\right|  \nonumber \\
    & =  \sum_{Y<N(\mathfrak{q})\leq \sqrt{X}} \;\sum_{N(\mathfrak{m}) \leq \frac{X}{Y^2N(\mathfrak{q})}}\left| \sum_{\substack{\max(Y,M_1(\mathfrak{m}))\leq N(\mathfrak{p})\leq  \min\left(N(\mathfrak{q}),\sqrt{\frac{X}{N(\mathfrak{mq})}}\right)            }       }      1    - \int_{\max(Y,M_1(\mathfrak{m}))}^{\min\left(N(\mathfrak{q}),\sqrt{\frac{X}{N(\mathfrak{mq})}}\right)}\frac{dt}{\log t}    \right|.
\end{align}
We observe that the difference inside the absolute value in (3.33) is almost\footnote[9]{we say ``almost" due to the truncated intervals of $N(\mathfrak{p})$} exactly the LHS of the strong form of the Prime Ideal Theorem (\textit{Theorem 1.3}). Note that the error term in \textit{Theorem 1.3}, i.e. $\frac{X}{e^{c_1\sqrt{\log X}}}$ is an increasing function in $X$. Therefore, continuing from (3.33), we can write using \textit{Theorem 1.3} that
\begin{align}
    |E| &\ll \sum_{Y<N(\mathfrak{q})\leq \sqrt{X}} \;\sum_{N(\mathfrak{m}) \leq \frac{X}{Y^2N(\mathfrak{q})}} \sqrt{\frac{X}{N(\mathfrak{m})N(\mathfrak{q})}}\frac{1}{e^{\frac{c_1}{\sqrt{2}}\sqrt{\log\left(\frac{X}{N(\mathfrak{m})N(\mathfrak{q})}\right)}}}\nonumber \\
    &\ll \frac{\sqrt{X}}{e^{c_1\sqrt{\log Y}}} \sum_{Y<N(\mathfrak{q})\leq \sqrt{X}}\frac{1}{\sqrt{N(\mathfrak{q})}} \sum_{N(\mathfrak{m}) \leq \frac{X}{Y^2N(\mathfrak{q})}} \frac{1}{\sqrt{N(\mathfrak{m})}} \nonumber \\
    &\ll \frac{X}{e^{c_1\sqrt{\log Y}}}\sum_{Y < N(\mathfrak{q})\leq \sqrt{X}} \frac{1}{N(\mathfrak{q})} \ll \frac{X\log \log X}{e^{c_1\sqrt{\log Y}}} .
\end{align}
We can treat the error $\Tilde{E}$ in an exactly similar manner. We write
\begin{align}
     |\Tilde{E}| &= \left| \sum_{\sqrt{X}<N(\mathfrak{q})\leq\frac{X}{Y^2}}\left(\sum_{\substack{Y<N(\mathfrak{p})\leq \sqrt{\frac{X}{N(\mathfrak{q})} } }} \Psi_K\left(\frac{X}{N(\mathfrak{p})^2N(\mathfrak{q})},N(\mathfrak{p})\right) - \int_Y^{\sqrt{\frac{X}{N(\mathfrak{q})} } } \Psi_K\left(\frac{X}{t^2N(\mathfrak{q})},t\right)\frac{dt}{\log t}\right)\right|  \nonumber \\
    &= \left| \sum_{Y<N(\mathfrak{q})\leq \frac{X}{Y^2}}\left(\sum_{\substack{Y<N(\mathfrak{p})\leq \sqrt{\frac{X}{N(\mathfrak{q})} } }}\; \sum_{\substack{N(\mathfrak{m}) \leq \frac{X}{N(\mathfrak{p})^2N(\mathfrak{q})} \\ M_1(\mathfrak{m})\leq N(\mathfrak{p})}} 1- \int_Y^{\sqrt{\frac{X}{N(\mathfrak{q})} } } \left\{\sum_{\substack{N(\mathfrak{m}) \leq \frac{X}{t^2N(\mathfrak{q})} \\ M_1(\mathfrak{m})\leq t}}1\right\}\frac{dt}{\log t}\right)\right| \nonumber \\
    &= \sum_{Y<N(\mathfrak{q})\leq \frac{X}{Y^2}} \sum_{N(\mathfrak{m}) \leq \frac{X}{Y^2N(\mathfrak{q})}}\left|           \sum_{\substack{\max(Y,M_1(\mathfrak{m}))\leq N(\mathfrak{p})\leq  \min\left(\sqrt{\frac{X}{N(\mathfrak{q})}},\sqrt{\frac{X}{N(\mathfrak{mq})}}\right) }       }      1    -\int_{\max(Y,M_1(\mathfrak{m}))}^{\min\left(\sqrt{\frac{X}{N(\mathfrak{q})}},\sqrt{\frac{X}{N(\mathfrak{mq})}}\right)}\frac{dt}{\log t}          \right| \nonumber \\
    &\ll \sum_{Y<N(\mathfrak{q})\leq \frac{X}{Y^2}} \sum_{N(\mathfrak{m}) \leq \frac{X}{Y^2N(\mathfrak{q})}} \sqrt{\frac{X}{N(\mathfrak{m})N(\mathfrak{q})}}\frac{1}{e^{\frac{c_1}{\sqrt{2}}\sqrt{\log\left(\frac{X}{N(\mathfrak{m})N(\mathfrak{q})}\right)}}} \nonumber \\
    &\ll \frac{\sqrt{X}}{e^{c_1\sqrt{\log Y}}}\sum_{Y<N(\mathfrak{q})\leq \frac{X}{Y^2}} \frac{1}{\sqrt{N(\mathfrak{q})}} \sum_{N(\mathfrak{m}) \leq \frac{X}{Y^2N(\mathfrak{q})}} \frac{1}{\sqrt{N(\mathfrak{m})}} \nonumber \\
    &\ll \frac{X}{e^{c_1\sqrt{\log Y}}} \sum_{Y<N(\mathfrak{q})\leq \frac{X}{Y^2}}\frac{1}{N(\mathfrak{q})} \ll \frac{X\log\log X}{e^{c_1\sqrt{\log Y}}}.
\end{align}
Therefore, combining (3.32), (3.34), and (3.35), we get that
\begin{align}
     \sum_{\substack{Y<N(\mathfrak{p}) \leq \sqrt{X} }} \left|S_2\left(\frac{X}{N(\mathfrak{p})},\mathfrak{p}\right)\right| = &\sum_{Y<N(\mathfrak{q})\leq \sqrt{X}} \int_Y^{N(\mathfrak{q})} \Psi_K\left(\frac{X}{t^2N(\mathfrak{q})},t\right)\frac{dt}{\log t}\nonumber \\ &+ \sum_{\sqrt{X}<N(\mathfrak{q})\leq \frac{X}{Y^2}} \int_Y^{\sqrt{\frac{X}{N(\mathfrak{q})} } } \Psi_K\left(\frac{X}{t^2N(\mathfrak{q})},t\right)\frac{dt}{\log t} +O\left(\frac{X\log\log X}{e^{c_1\sqrt{\log Y}}}\right).
\end{align}
Now, we need to estimate the pseudo-integrals present in the RHS of (3.36). To do that, we will need the estimate of $\Psi_K(X,Y)$ as given in (3.3). Let us start by looking at the integral
\begin{eqnarray}
    \int_Y^{N(\mathfrak{q})} \Psi_K\left(\frac{X}{t^2N(\mathfrak{q})},t\right)\frac{dt}{\log t} \ll \int_Y^{N(\mathfrak{q})} \frac{X}{t^2N(\mathfrak{q})} e^{-\beta\log\left(\frac{\beta}{e}\right)}\frac{dt}{\log t},   
\end{eqnarray} 
where $\beta = \frac{\log \frac{X}{t^2N(\mathfrak{q})}}{\log t}$. In the integral, we make a substitution of $u = \frac{\log \frac{X}{N(\mathfrak{q})}}{\log t}$. Then,
\begin{align}
     \beta = u-2;\quad t = e^{\frac{\log \frac{X}{N(\mathfrak{q})}}{u}};\quad  du = -\log \frac{X}{N(\mathfrak{q})}\frac{dt}{t(\log t)^2}. \nonumber
\end{align}
Note here that the asymptotic bound is uniform for $Y<t\leq N(\mathfrak{q})$, for $\varepsilon=\frac{2}{3}$. Therefore, continuing from (3.37)
 \begin{align}
     \int_Y^{N(\mathfrak{q})} \Psi_K\left(\frac{X}{t^2N(\mathfrak{q})},t\right)\frac{dt}{\log t} \ll \frac{X}{N(\mathfrak{q})}\int_2^{\frac{\log \frac{X}{N(\mathfrak{q})}}{\log Y}} e^{-\frac{\log \frac{X}{N(\mathfrak{q})}}{u}-(u-2)\log(u-2)}\frac{du}{u} \ll \frac{X}{N(\mathfrak{q})Y}.
 \end{align}
Hence, the pseudo-integro-sum can then be estimated as
\begin{eqnarray}
    \sum_{Y<N(\mathfrak{q})\leq \sqrt{X}} \int_Y^{N(\mathfrak{q})} \Psi_K\left(\frac{X}{t^2N(\mathfrak{q})},t\right)\frac{dt}{\log t} \ll \sum_{Y<N(\mathfrak{q}) \leq \sqrt{X}} \frac{X}{N(\mathfrak{q})Y} \ll \frac{X\log\log X}{Y}.
\end{eqnarray}
For the other inner integral on the right of (3.36), a similar treatment yields
\begin{align}
    \int_Y^{\sqrt{\frac{X}{N(\mathfrak{q})} } } \Psi_K\left(\frac{X}{t^2N(\mathfrak{q})},t\right)\frac{dt}{\log t} \ll \int_{Y}^{\sqrt{\frac{X}{N(\mathfrak{q})}}} \frac{X}{t^2N(\mathfrak{q})}e^{-\beta\log\frac{\beta}{e}}\frac{dt}{\log t}.
\end{align}
Using the exact same substitution as done above, we get from (3.40) that
\begin{align}
    \int_Y^{\sqrt{\frac{X}{N(\mathfrak{q})} } } \Psi_K\left(\frac{X}{t^2N(\mathfrak{q})},t\right)\frac{dt}{\log t} \ll \frac{X}{N(\mathfrak{q})}\int_2^{\frac{\log \frac{X}{N(\mathfrak{q})}}{\log Y}} e^{-\frac{\log \frac{X}{N(\mathfrak{q})}}{u}-(u-2)\log(u-2)}\frac{du}{u} \ll \frac{X}{N(\mathfrak{q})Y}.
\end{align}
Therefore, again, the second pseudo-integro-sum in (3.36) can be estimated as
\begin{align}
     \sum_{\sqrt{X}<N(\mathfrak{q})\leq \frac{X}{Y^2}} \int_Y^{\sqrt{\frac{X}{N(\mathfrak{q})} } } \Psi_K\left(\frac{X}{t^2N(\mathfrak{q})},t\right)\frac{dt}{\log t} \ll \frac{X\log\log X}{Y}.
\end{align}
Therefore, finally combining (3.36), (3.39), and (3.42), we get 
\begin{align}
     \sum_{\substack{Y<N(\mathfrak{p}) \leq \sqrt{X} }} \left|S_2\left(\frac{X}{N(\mathfrak{p})},\mathfrak{p}\right)\right| \ll \frac{X\log\log X}{Y}+\frac{X\log\log X}{e^{c_1\sqrt{\log Y}}} \ll \frac{X\log\log X}{e^{c_1\sqrt{\log Y}}}.
\end{align}
Thus, combining (3.28), (3.30), and (3.43) together yields our required result.   
\end{proof}

\begin{remark}
    Note that we have deliberately kept two different error terms dependent on the choice of $Y$. We will make a suitable choice of $Y$ in the proof of \textit{Theorem 4.1} (see below), where we will use \textit{Lemma 3.4}.
\end{remark}
\noindent We have now stated and proved all the lemmas, and hence have the required machinery that we will need to prove the main results in the next section. It is evident through the proofs that both the estimates of $\Psi_K(X,Y)$ mentioned at the start of this section are important in proving these lemmas. We now proceed to state and prove our major theorems in the next section that will lead us to derive a new formula for the Chebotarev Density, with a special focus on the second order duality between prime ideals.

\section{A New Formula for the Chebotarev Density in arbitrary finite Galois extensions}
We will prove two theorems in this section, leading up to a new formula for the Chebotarev Density. We will be proving the results in a setting of a finite Galois extension $L/K$ such that $G=Gal(L/K)$ is the corresponding Galois group. We fix a conjugacy class $C\subset G$ which will be instrumental in our proofs. As mentioned above, the following results can be viewed as generalizations of the works of Alladi, Dawsey, Sweeting, Woo, Johnson, and the author himself. More precisely, the result that will be derived at the end of this section will prove to be a perfect generalization\footnote[10]{our result generalizes the choice of the Galois extension} to equation (6.3) in [Se25]. \\ \\
Our first theorem in this section is a crucial one for our goal. It enables us to bring in the second order duality (\textit{Theorem 2.3}) and the Chebotarev Density into the picture. This theorem can be treated as the backbone of this whole process to deduce the formula. 
\begin{theorem}
    Let $L/K$ be a finite Galois extension with $G=Gal(L/K)$ and let $C \subset G$ be a fixed conjugacy class. Then for $X\geq 2$, we have
    \begin{align}
         \sum_{2 \leq N(I)\leq X} Q^2_C(I) = c_K\cdot \frac{|C|}{|G|}\cdot X + O\left(\frac{X(\log\log X)^2}{\log X}\right). \nonumber
    \end{align}
    \end{theorem}
    \begin{proof}
        Let us first recall the definition of $Q^2_C(I)$.
        \begin{align}
             Q_C^2(I) &:= \# \left\{\mathfrak{p} \subset \mathcal{O}_K :\; \mathfrak{p}\text{ is a prime; } I\subset \mathfrak{p}; \;M_2(I)=N(\mathfrak{p})\;\text{and}\;\left[\frac{L/K}{\mathfrak{p}}\right]=C,\;\mathfrak{p}\;\text{is unramified}\right\} .\nonumber
        \end{align}
        By definition, it is clear that ideals in $\mathcal{O}_K$ contained in prime ideals only of a particular norm do not contribute to the sum in the LHS. As done in the proof of \textit{Lemma 3.4}, in a similar fashion, we define the set $S_2(X,\mathfrak{p})$ as the set of ideals $I$ in $\mathcal{O}_K$ such that $N(I) \leq X$ with a unique and non-repeating prime ideal of the largest norm and $M_2(I)=N(\mathfrak{p})$. We note here that for any ideal $I \subset \mathfrak{p}$, for some prime $\mathfrak{p}$ such that $N(I)\leq X$ and $M_2(I)=N(\mathfrak{p})$, $N(\mathfrak{p}) \leq \sqrt{X}$. Therefore, using \textit{Lemma 3.1}, we have that
        \begin{align}
             \sum_{2\leq N(I)\leq X} Q_C^2(I) = \sum_{\substack{N(\mathfrak{p})\leq \sqrt{X} \\ \left[\frac{L/K}{\mathfrak{p}}\right]=C}} |S_2(X,\mathfrak{p})| +  O\left(\frac{X}{e^{c_3\sqrt{\log X\log\log X}}}\right).
        \end{align}
        We therefore move on to estimate the sum on the RHS of (4.1). We choose $I \in S_2(X,\mathfrak{p})$ such that $I=\mathfrak{M.pq}$, where $M_1(I)=\mathfrak{q}$, $M_2(I)=\mathfrak{p}$ and $\mathfrak{M} \subset \mathcal{O}_K$ is an ideal such that $M_1(\mathfrak{M})\leq N(\mathfrak{p})$. Thus, we write
        \begin{align}
            |S_2(X,\mathfrak{p})| = \sum_{\substack{N(\mathfrak{M}) \leq \frac{X}{N(\mathfrak{p})^2} \\ M_1(\mathfrak{M}) \leq N(\mathfrak{p})}} \sum_{\substack{N(\mathfrak{p})<N(\mathfrak{q}) \\ N(\mathfrak{M.pq}) \leq X}} 1 = \sum_{N(\mathfrak{p})<N(\mathfrak{q})\leq \frac{X}{N(\mathfrak{p})}} \sum_{\substack{N(\mathfrak{M})\leq \frac{X}{N(\mathfrak{pq})} \\ M_1(\mathfrak{M}) \leq N(\mathfrak{p})}} 1 .
        \end{align}
    Using the definition of $\Psi_K(X,Y)$, we then have that 
    \begin{align}
        |S_2(X,\mathfrak{p})| = \sum_{\substack{\mathfrak{q} \in \mathcal{O}_K \\ \mathfrak{q} \text{ is a prime} \\N(\mathfrak{p})<N(\mathfrak{q})\leq \frac{X}{N(\mathfrak{p})}}} \Psi_K\left(\frac{X}{N(\mathfrak{pq})},N(\mathfrak{p})\right).
    \end{align} 
    Now, we observe that for $Y \leq e^{(\log X)^{1-\delta}}$, for some $\delta  >0$, we have using (3.27) that
    \begin{align}
        \sum_{\substack{N(\mathfrak{p}) \leq Y \\ \left[\frac{L/K}{\mathfrak{p}}\right] =C}} |S_2(X,\mathfrak{p})| \leq \sum_{N(\mathfrak{p}) \leq Y} |S_2(X,\mathfrak{p})| \leq   \sum_{N(\mathfrak{p})\leq Y} \Psi_{K,2}\left({X},\mathfrak{p}\right) \ll \frac{X\log Y}{\log X}.
    \end{align}
        Therefore, we are left to estimate the sum of $|S_2(X,\mathfrak{p})|$ over the interval $Y<N(\mathfrak{p})\leq \sqrt{X}$. Therefore, from (4.3) we have
        \begin{align}
            \sum_{\substack{Y<N(\mathfrak{p}) \leq \sqrt{X} \\ \left[\frac{L/K}{\mathfrak{p}}\right] =C}} |S_2(X,\mathfrak{p})| = \sum_{\substack{Y<N(\mathfrak{p}) \leq \sqrt{X} \\ \left[\frac{L/K}{\mathfrak{p}}\right] =C}}\; \sum_{N(\mathfrak{p})<N(\mathfrak{q})\leq \frac{X}{N(\mathfrak{p})}} \Psi_K\left(\frac{X}{N(\mathfrak{p}\mathfrak{q})},N(\mathfrak{p})\right).
        \end{align}
        It is worth mentioning here that the idea of this proof is quite similar to the proof of \textit{Lemma 3.4}. We change the order of summation in the RHS of (4.5), which would require a split in the summation. We keep in mind the following conditions that $N(\mathfrak{p})$ need to satisfy:
        \begin{align}
            Y<N(\mathfrak{p}), \quad N(\mathfrak{p})<N(\mathfrak{q}), \quad N(\mathfrak{p})\leq \frac{X}{N(\mathfrak{q})},\quad N(\mathfrak{p}) \leq \sqrt{X} .\nonumber
        \end{align}
        Continuing from (4.5), we then get 
        \begin{align}
            \sum_{\substack{Y<N(\mathfrak{p}) \leq \sqrt{X} \\ \left[\frac{L/K}{\mathfrak{p}}\right] =C}} |S_2(X,\mathfrak{p})| &= \sum_{\substack{Y <N(\mathfrak{q}) \leq \sqrt{X}}} \;  \sum_{\substack{Y<N(\mathfrak{p})<N(\mathfrak{q}) \\ \left[\frac{L/K}{\mathfrak{p}}\right] =C}} \Psi_K\left(\frac{X}{N(\mathfrak{p}\mathfrak{q})},N(\mathfrak{p})\right) \nonumber \\ &\hspace{2cm}+ \sum_{\sqrt{X} < N(\mathfrak{q}) \leq \frac{X}{Y}} \;\sum_{\substack{Y<N(\mathfrak{p})\leq \frac{X}{N(\mathfrak{q})} \\ \left[\frac{L/K}{\mathfrak{p}}\right] =C} } \Psi_K\left(\frac{X}{N(\mathfrak{p}\mathfrak{q})},N(\mathfrak{p})\right) .
        \end{align}
        To bring in the Chebotarev Density, we now interchange the inner sums of both the double sums in the RHS of (4.6) with suitable integrals. This is indeed one of the most crucial steps of the whole proof. We do the replacements as follows:
        \begin{align}
            \sum_{\substack{Y<N(\mathfrak{p})<N(\mathfrak{q}) \\ \left[\frac{L/K}{\mathfrak{p}}\right] =C}} \Psi_K\left(\frac{X}{N(\mathfrak{p}\mathfrak{q})},N(\mathfrak{p})\right) \quad &\text{by} \quad \frac{|C|}{|G|}\int_Y^{N(\mathfrak{q})} \Psi_K\left(\frac{X}{tN(\mathfrak{q})},t\right)\frac{dt}{\log t} ,\nonumber \\
     \sum_{\substack{Y<N(\mathfrak{p})<\frac{X}{N(\mathfrak{q})} \\ \left[\frac{L/K}{\mathfrak{p}}\right] =C}} \Psi_K\left(\frac{X}{N(\mathfrak{p}\mathfrak{q})},N(\mathfrak{p})\right) \quad &\text{by} \quad \frac{|C|}{|G|}\int_Y^{\frac{X}{N(\mathfrak{q})}} \Psi_K\left(\frac{X}{tN(\mathfrak{q})},t\right)\frac{dt}{\log t}. \nonumber
        \end{align}
         Of course, it is now time to estimate the following errors that arise due to such a replacement:
    \begin{align}
        E_1&:=\sum_{Y<N(\mathfrak{q})\leq \sqrt{X}}\left(\sum_{\substack{Y<N(\mathfrak{p})<N(\mathfrak{q}) \\ \left[\frac{L/K}{\mathfrak{p}}\right] =C}} \Psi_K\left(\frac{X}{N(\mathfrak{p}\mathfrak{q})},N(\mathfrak{p})\right) - \frac{|C|}{|G|}\int_Y^{N(\mathfrak{q})} \Psi_K\left(\frac{X}{tN(\mathfrak{q})},t\right)\frac{dt}{\log t}\right) ,\nonumber \\
     E_2 &:=\sum_{\sqrt{X}<N(\mathfrak{q})\leq \frac{X}{Y}}\left(\sum_{\substack{Y<N(\mathfrak{p})<\frac{X}{N(\mathfrak{q})} \\ \left[\frac{L/K}{\mathfrak{p}}\right] =C}} \Psi_K\left(\frac{X}{N(\mathfrak{p}\mathfrak{q})},N(\mathfrak{p})\right)- \frac{|C|}{|G|}\int_Y^{\frac{X}{N(\mathfrak{q})}} \Psi_K\left(\frac{X}{tN(\mathfrak{q})},t\right)\frac{dt}{\log t}\right) .\nonumber
    \end{align}
From (4.6), we then have 
\begin{align}
    \sum_{\substack{Y<N(\mathfrak{p}) \leq \sqrt{X} \\ \left[\frac{L/K}{\mathfrak{p}}\right] =C}} |S_2(X,\mathfrak{p})| = \frac{|C|}{|G|}\sum_{Y<N(\mathfrak{q})\leq \sqrt{X}}\int_Y^{N(\mathfrak{q})} \Psi_K\left(\frac{X}{tN(\mathfrak{q})},t\right)\frac{dt}{\log t}& \nonumber \\+
\frac{|C|}{|G|}\sum_{\sqrt{X}<N(\mathfrak{q})\leq \frac{X}{Y}}\int_Y^{\frac{X}{N(\mathfrak{q})}}& \Psi_K\left(\frac{X}{tN(\mathfrak{q})},t\right)\frac{dt}{\log t}\nonumber \\&+E_1+E_2 .
\end{align}
To estimate the error terms $E_1$ and $E_2$, we use the strong form of Chebotarev Density Theorem (\textit{Theorem 1.2}). Changing the order of summations in $E_1$, we observe that
\begin{align}
    |E_1| &= \left| \sum_{Y<N(\mathfrak{q})\leq \sqrt{X}}\left(\sum_{\substack{Y<N(\mathfrak{p})<N(\mathfrak{q}) \\ \left[\frac{L/K}{\mathfrak{p}}\right] =C}} \Psi_K\left(\frac{X}{N(\mathfrak{p}\mathfrak{q})},N(\mathfrak{p})\right) - \frac{|C|}{|G|}\int_Y^{N(\mathfrak{q})} \Psi_K\left(\frac{X}{tN(\mathfrak{q})},t\right)\frac{dt}{\log t}\right)\right|  \nonumber \\
    &= \left| \sum_{Y<N(\mathfrak{q})\leq \sqrt{X}}\left(\sum_{\substack{Y<N(\mathfrak{p})<N(\mathfrak{q}) \\ \left[\frac{L/K}{\mathfrak{p}}\right] =C}}\; \sum_{\substack{N(\mathfrak{M}) \leq \frac{X}{N(\mathfrak{pq})} \\ M_1(\mathfrak{M})\leq N(\mathfrak{p})}} 1- \frac{|C|}{|G|}\int_Y^{N(\mathfrak{q})} \left\{\sum_{\substack{N(\mathfrak{M}) \leq \frac{X}{tN(\mathfrak{q})} \\ M_1(\mathfrak{M})\leq t}}1\right\}\frac{dt}{\log t}\right)\right| \nonumber \\
    &=    \sum_{Y<N(\mathfrak{q})\leq \sqrt{X}} \;\sum_{N(\mathfrak{M}) \leq \frac{X}{YN(\mathfrak{q})}}\left| \sum_{\substack{\max(Y,M_1(\mathfrak{M}))\leq N(\mathfrak{p})\leq  \min\left(N(\mathfrak{q}),\frac{X}{N(\mathfrak{Mq})}\right) \\  \left[\frac{L/K}{\mathfrak{p}}\right] =C            }       }      1    - \frac{|C|}{|G|}\int_{\max(Y,M_1(\mathfrak{M}))}^{\min\left(N(\mathfrak{q}),\frac{X}{N(\mathfrak{Mq})}\right)}\frac{dt}{\log t}    \right|.
\end{align}
Of course, the innermost summand involving the absolute value resembles the LHS of \textit{Theorem 1.2}. Also, we note that error function $\frac{X}{e^{c_1\sqrt{\frac{\log X}{n_K}}}}$ is increasing, and therefore, we use the CDT with the bound $\frac{X}{N(\mathfrak{Mq})}$ on $N(\mathfrak{p})$. We denote the constant $c_4 = \frac{c_1}{\sqrt{n_K}}$. Thus, from (4.8), we get
\begin{align}
     |E_1| &\ll \sum_{Y<N(\mathfrak{q})\leq \sqrt{X}} \;\sum_{N(\mathfrak{M}) \leq \frac{X}{YN(\mathfrak{q})}} \frac{X}{N(\mathfrak{Mq})e^{c_4\sqrt{\log \left(\frac{X}{N(\mathfrak{Mq})}\right)}}} \nonumber \\
    &\ll \frac{X}{e^{c_4\sqrt{\log Y}}}\sum_{Y<N(\mathfrak{q})\leq \sqrt{X}}\frac{1}{N(\mathfrak{q})} \;\sum_{N(\mathfrak{M}) \leq \frac{X}{YN(\mathfrak{q})}}\frac{1}{N(\mathfrak{M})} \nonumber \\ 
    &\ll \frac{X\log X\log\log X}{e^{c_4\sqrt{\log Y}}}.
\end{align}
Similarly, for $E_2$, we have that
\begin{align}
    |E_2| &= \left| \sum_{\sqrt{X}<N(\mathfrak{q})\leq\frac{X}{Y}}\left(\sum_{\substack{Y<N(\mathfrak{p})\leq \frac{X}{N(\mathfrak{q})} \\ \left[\frac{L/K}{\mathfrak{p}}\right] =C}} \Psi_K\left(\frac{X}{N(\mathfrak{p}\mathfrak{q})},N(\mathfrak{p})\right) - \frac{|C|}{|G|}\int_Y^{\frac{X}{N(\mathfrak{q})}} \Psi_K\left(\frac{X}{tN(\mathfrak{q})},t\right)\frac{dt}{\log t}\right)\right|  \nonumber \\
    &= \left| \sum_{Y<N(\mathfrak{q})\leq \frac{X}{Y}}\left(\sum_{\substack{Y<N(\mathfrak{p})\leq \frac{X}{N(\mathfrak{q})} \\ \left[\frac{L/K}{\mathfrak{p}}\right] =C}}\; \sum_{\substack{N(\mathfrak{M}) \leq \frac{X}{N(\mathfrak{pq})} \\ M_1(\mathfrak{M})\leq N(\mathfrak{p})}} 1- \frac{|C|}{|G|}\int_Y^{\frac{X}{N(\mathfrak{q})}} \left\{\sum_{\substack{N(\mathfrak{M}) \leq \frac{X}{tN(\mathfrak{q})} \\ M_1(\mathfrak{M})\leq t}}1\right\}\frac{dt}{\log t}\right)\right| \nonumber \\
    &= \sum_{Y<N(\mathfrak{q})\leq \frac{X}{Y}} \sum_{N(\mathfrak{M}) \leq \frac{X}{YN(\mathfrak{q})}}\left|           \sum_{\substack{\max(Y,M_1(\mathfrak{M}))\leq N(\mathfrak{p})\leq  \min\left(\frac{X}{N(\mathfrak{q})},\frac{X}{N(\mathfrak{Mq})}\right) \\  \left[\frac{L/K}{\mathfrak{p}}\right] =C            }       }      1    - \frac{|C|}{|G|}\int_{\max(Y,M_1(\mathfrak{M}))}^{\min\left(\frac{X}{N(\mathfrak{q})},\frac{X}{N(\mathfrak{Mq})}\right)}\frac{dt}{\log t}          \right| \nonumber \\
     &\ll \sum_{\sqrt{X}<N(\mathfrak{q}) \leq \frac{X}{Y}}\sum_{N(\mathfrak{M}) \leq \frac{X}{YN(\mathfrak{q})}} \frac{X}{N(\mathfrak{Mq})e^{c_4\sqrt{\log\left(\frac{X}{N(\mathfrak{Mq})}\right)}}} \nonumber \\
    &\ll \frac{X}{e^{c_4\sqrt{\log Y}}}\sum_{\sqrt{X}<N(\mathfrak{q})\leq \frac{X}{Y}}\frac{1}{N(\mathfrak{q})}\sum_{N(\mathfrak{M})\leq \frac{X}{YN(\mathfrak{q})}}\frac{1}{N(\mathfrak{M})} \nonumber \\
    &\ll \frac{X\log X\log \log X}{e^{c_4\sqrt{\log Y}}}.
\end{align}
Therefore, combining (4.7), (4.9), and (4.10), we have that
\begin{align}
     \sum_{\substack{Y<N(\mathfrak{p}) \leq \sqrt{X} \\ \left[\frac{L/K}{\mathfrak{p}}\right] =C}} |S_2(X,\mathfrak{p})| = \frac{|C|}{|G|}&\sum_{Y<N(\mathfrak{q})\leq \sqrt{X}}\int_Y^{N(\mathfrak{q})} \Psi_K\left(\frac{X}{tN(\mathfrak{q})},t\right)\frac{dt}{\log t} \nonumber \\+
&\frac{|C|}{|G|}\sum_{\sqrt{X}<N(\mathfrak{q})\leq \frac{X}{Y}}\int_Y^{\frac{X}{N(\mathfrak{q})}} \Psi_K\left(\frac{X}{tN(\mathfrak{q})},t\right)\frac{dt}{\log t}+O\left(\frac{X\log X\log \log X}{e^{c_4\sqrt{\log Y}}}\right).
\end{align}
Therefore, along with (4.4), we get from (4.11) that
\begin{align}
     \sum_{\substack{N(\mathfrak{p}) \leq \sqrt{X} \\ \left[\frac{L/K}{\mathfrak{p}}\right] =C}} |S_2(X,\mathfrak{p})| = \frac{|C|}{|G|}\left[\sum_{Y<N(\mathfrak{q})\leq \sqrt{X}}\right.&\left.\int_Y^{N(\mathfrak{q})} \Psi_K\left(\frac{X}{tN(\mathfrak{q})},t\right)\frac{dt}{\log t} +
\sum_{\sqrt{X}<N(\mathfrak{q})\leq \frac{X}{Y}}\int_Y^{\frac{X}{N(\mathfrak{q})}} \Psi_K\left(\frac{X}{tN(\mathfrak{q})},t\right)\frac{dt}{\log t}\right]\nonumber \\&+O\left(\frac{X\log Y}{\log X}\right)+O\left(\frac{X\log X\log \log X}{e^{c_4\sqrt{\log Y}}}\right).
\end{align}
Now, we will evaluate the pseudo-integro-sums using further estimates. Surprisingly, the very same idea used above will work here too. Let us recall the definition of $Q^2(I)$:
\begin{align}
    Q^2(I)&:=\# \{\mathfrak{p}\subset  \mathcal{O}_K:\;I \subset \mathfrak{p};\;M_2(I)=Nm(\mathfrak{p})\} . \nonumber
\end{align}
Note that by the definition of $S_2(X,\mathfrak{p})$, we can write using \textit{Lemma 3.1}, in a similar fashion as of (4.1) that
\begin{align}
    \sum_{2\leq N(I)\leq X} Q^2(I) = \sum_{N(\mathfrak{p})\leq \sqrt{X} } |S_2(X,\mathfrak{p})| +  O\left(\frac{X}{e^{c_3\sqrt{\log X\log\log X}}}\right).
\end{align}
Here, we observe that the sum on the right of (4.13) is exactly the same as that we treated in the first half of the proof, without the conjugacy class condition on the prime ideal $\mathfrak{p}$. But the change in the consideration of the condition hardly changes anything in the estimation of the sum. We can start by using the estimate in (3.27) to see that for $Y\leq e^{(\log X)^{1-\delta}}$, for some $\delta>0$, 
\begin{align}
    \sum_{N(\mathfrak{p})\leq Y} |S_2(X,\mathfrak{p})| \ll\frac{X\log Y}{\log X}.
\end{align}
Using the idea in (4.3), we can write here that
\begin{align}
    \sum_{Y<N(\mathfrak{p}) \leq \sqrt{X}} |S_2(X,\mathfrak{p})| = \sum_{\substack{Y<N(\mathfrak{p}) \leq \sqrt{X} }} \sum_{N(\mathfrak{p})<N(\mathfrak{q})\leq \frac{X}{N(\mathfrak{p})}} \Psi_K\left(\frac{X}{N(\mathfrak{p}\mathfrak{q})},N(\mathfrak{p})\right) .
\end{align}
As done earlier, we change the order of summation in (4.15) and then split the double sum into two to get
\begin{align}
    \sum_{\substack{Y<N(\mathfrak{p}) \leq \sqrt{X} }} |S_2(X,\mathfrak{p})| &= \sum_{\substack{Y <N(\mathfrak{q}) \leq \sqrt{X}}} \;  \sum_{\substack{Y<N(\mathfrak{p})<N(\mathfrak{q}) }} \Psi_K\left(\frac{X}{N(\mathfrak{p}\mathfrak{q})},N(\mathfrak{p})\right) \nonumber \\ &\hspace{2cm}+ \sum_{\sqrt{X} < N(\mathfrak{q}) \leq \frac{X}{Y}} \;\sum_{\substack{Y<N(\mathfrak{p})\leq \frac{X}{N(\mathfrak{q})}} } \Psi_K\left(\frac{X}{N(\mathfrak{p}\mathfrak{q})},N(\mathfrak{p})\right) .
\end{align}
We observe that the only difference between equations (4.6) and (4.16) is the conjugacy class condition. Thus, we do a similar replacement of the inner sums with integrals, but this time, without the Chebotarev Density factor in front of the integrals. We make the following replacements:
\begin{align}
    \sum_{\substack{Y<N(\mathfrak{p})<N(\mathfrak{q}) }} \Psi_K\left(\frac{X}{N(\mathfrak{p}\mathfrak{q})},N(\mathfrak{p})\right) \quad &\text{by} \quad \int_Y^{N(\mathfrak{q})} \Psi_K\left(\frac{X}{tN(\mathfrak{q})},t\right)\frac{dt}{\log t}, \nonumber \\
     \sum_{\substack{Y<N(\mathfrak{p})<\frac{X}{N(\mathfrak{q})} }} \Psi_K\left(\frac{X}{N(\mathfrak{p}\mathfrak{q})},N(\mathfrak{p})\right) \quad &\text{by} \quad \int_Y^{\frac{X}{N(\mathfrak{q})}} \Psi_K\left(\frac{X}{tN(\mathfrak{q})},t\right)\frac{dt}{\log t}. \nonumber
\end{align}
Of course, we then have the errors that need to be estimated. We denote them as:
\begin{align}
    E_3&:=\sum_{Y<N(\mathfrak{q})\leq \sqrt{X}}\left(\sum_{\substack{Y<N(\mathfrak{p})<N(\mathfrak{q}) }} \Psi_K\left(\frac{X}{N(\mathfrak{p}\mathfrak{q})},N(\mathfrak{p})\right) - \int_Y^{N(\mathfrak{q})} \Psi_K\left(\frac{X}{tN(\mathfrak{q})},t\right)\frac{dt}{\log t}\right), \nonumber \\
     E_4 &:=\sum_{\sqrt{X}<N(\mathfrak{q})\leq \frac{X}{Y}}\left(\sum_{\substack{Y<N(\mathfrak{p})<\frac{X}{N(\mathfrak{q})}}} \Psi_K\left(\frac{X}{N(\mathfrak{p}\mathfrak{q})},N(\mathfrak{p})\right)-\int_Y^{\frac{X}{N(\mathfrak{q})}} \Psi_K\left(\frac{X}{tN(\mathfrak{q})},t\right)\frac{dt}{\log t}\right). \nonumber
\end{align}
For the estimation of $E_3$ and $E_4$, we use \textit{Theorem 1.3} and other standard summation estimates, as used in the proof above, and also in the proof of \textit{Lemma 3.4}. Following similar modifications and calculations, we finally get that
\begin{align}
    |E_3| \ll \frac{X\log X \log\log X}{e^{c_1\sqrt{\log Y}}} \quad \text{and} \quad |E_4| \ll \frac{X\log X \log\log X}{e^{c_1\sqrt{\log Y}}}.
\end{align}
Therefore, combining (4.14), (4.16), and (4.17), we have that
\begin{align}
    \sum_{\substack{N(\mathfrak{p}) \leq \sqrt{X} }} |S_2(X,\mathfrak{p})| = \sum_{Y<N(\mathfrak{q})\leq \sqrt{X}}&\int_Y^{N(\mathfrak{q})} \Psi_K\left(\frac{X}{tN(\mathfrak{q})},t\right)\frac{dt}{\log t} +
\sum_{\sqrt{X}<N(\mathfrak{q})\leq \frac{X}{Y}}\int_Y^{\frac{X}{N(\mathfrak{q})}} \Psi_K\left(\frac{X}{tN(\mathfrak{q})},t\right)\frac{dt}{\log t}\nonumber \\&+O\left(\frac{X\log Y}{\log X}\right)+O\left(\frac{X\log X\log \log X}{e^{c_1\sqrt{\log Y}}}\right).
\end{align}
Further, using (4.13), we can write from (4.18) that
\begin{align}
    \sum_{2\leq N(I)\leq X} Q^2(I) = & \sum_{Y<N(\mathfrak{q})\leq \sqrt{X}}\int_Y^{N(\mathfrak{q})} \Psi_K\left(\frac{X}{tN(\mathfrak{q})},t\right)\frac{dt}{\log t} +
\sum_{\sqrt{X}<N(\mathfrak{q})\leq \frac{X}{Y}}\int_Y^{\frac{X}{N(\mathfrak{q})}} \Psi_K\left(\frac{X}{tN(\mathfrak{q})},t\right)\frac{dt}{\log t}\nonumber \\&+O\left(\frac{X\log Y}{\log X}\right)+O\left(\frac{X\log X\log \log X}{e^{c_1\sqrt{\log Y}}}\right)+  O\left(\frac{x}{e^{c_3\sqrt{\log X\log\log X}}}\right).
\end{align}
We can rewrite the LHS of (4.19), using \textit{Theorem 1.5}, as
\begin{align}
     \sum_{2\leq N(I)\leq X} Q^2(I) = \sum_{2 \leq N(I)\leq X}1+\sum_{\substack{2\leq N(I) \leq X \\ Q^2(I) \geq 2}} (Q^2(I)-1) = c_K\cdot X + \sum_{\substack{2\leq N(I) \leq X \\ Q^2(I) \geq 2}} (Q^2(I)-1)  + O\left(X^{1-\frac{1}{d}}\right),
\end{align}
where $d = [K:\mathbb{Q}]$. Then using \textit{Lemma 3.4}, we have that
\begin{align}
     \sum_{2\leq N(I)\leq X} Q^2(I) =
c_K\cdot X +O\left(\frac{X\log Y}{\log X}\right)+O\left(\frac{X\log\log X}{e^{c_1\sqrt{\log Y}}}\right) .
\end{align}
Thus, combining (4.19) and (4.21) and rearranging the terms, we get
\begin{align}
    &\sum_{Y<N(\mathfrak{q})\leq \sqrt{X}}\int_Y^{N(\mathfrak{q})} \Psi_K\left(\frac{X}{tN(\mathfrak{q})},t\right)\frac{dt}{\log t} +
\sum_{\sqrt{X}<N(\mathfrak{q})\leq \frac{X}{Y}}\int_Y^{\frac{X}{N(\mathfrak{q})}} \Psi_K\left(\frac{X}{tN(\mathfrak{q})},t\right)\frac{dt}{\log t} \nonumber \\ 
&= c_K\cdot X +O\left(\frac{X\log Y}{\log X}\right)+O\left(\frac{X\log X\log\log X}{e^{c_4\sqrt{\log Y}}}\right).
\end{align}
Interestingly, the LHS of (4.22) is exactly the expression in brackets in the RHS of (4.12) that we were required to estimate. Therefore, now combining (4.1), (4.12), and (4.22), we finally get that
\begin{align}
     \sum_{2\leq N(I)\leq X} Q_C^2(I) = c_K\cdot \frac{|C|}{|G|}\cdot X  +O\left(\frac{X\log Y}{\log X}\right)+O\left(\frac{X\log X\log\log X}{e^{c_4\sqrt{\log Y}}}\right).
\end{align}
  Thus, a suitable choice of $Y = e^{\left(\frac{2}{c_4}\log\log X\right)^2}$ gives our desired result.
    \end{proof}
 \begin{remark}
     We observe that the quantitative approach of the proof of \textit{Theorem 4.1} gives us an asymptotic relation, i.e.
     \begin{align}
         \sum_{2\leq N(I)\leq X} Q_C^2(I) \sim c_K\cdot \frac{|C|}{|G|}\cdot X,
     \end{align}
     along with an estimated error term.
 \end{remark}
\noindent 
\noindent Recalling \textit{Theorem 1.1} from [SW18], we see that it is also an asymptotic relation, and also, surprisingly, has the same right hand side. Precisely,
\begin{align}
    \sum_{2\leq N(I)\leq X} Q_C(I) \sim c_K\cdot \frac{|C|}{|G|}\cdot X.
\end{align}
 It is indeed intriguing to see that, irrespective of the $k^{th}$ largest norm that we choose, the sum on the left has the same asymptote due to the use of the Chebotarev Density Theorem. One can also prove the same asymptotic result for other $k^{th}$ largest norms, as done in the classical case for any arbitrary $k^{th}$ largest prime factors (see [AS(ip)]). We now prove the main theorem that provides us with the formula for the Chebotarev Density using the first and second order duality.  
\begin{theorem}
     Let $f$ be a characteristic function of the set $S(L/K;C)$. For a finite Galois extension $L/K$ and a fixed conjugacy class $C\subset G=Gal(L/K)$, we have that 
     \begin{align}
         \lim_{X\to \infty}\sum_{\substack{2\leq N(I) \leq X \\ I \in S(L/K;C)}}\frac{\mu_K(I)(\omega_K(I)-1)}{N(I)}= \lim_{X\to \infty}\sum_{2\leq N(I) \leq X }\frac{\mu_K(I)(\omega_K(I)-1)f(I)}{N(I)} = \frac{|C|}{|G|} .
         \nonumber 
     \end{align}
\end{theorem}
\begin{proof}
    We start by considering the sum on the left multiplied by $X$. We have
    \begin{align}
        X\sum_{\substack{2\leq N(I) \leq X \\ I \in S(L/K;C)}}\frac{\mu_K(I)\omega_K(I)}{N(I)} &= \sum_{\substack{2\leq N(I) \leq X \\ I \in S(L/K;C)}} \mu_K(I)\omega_K(I)\cdot\frac{X}{N(I)} \nonumber \\
        &= \frac{1}{c_K}\sum_{\substack{2\leq N(I) \leq X \\ I \in S(L/K;C)}} \mu_K(I)\omega_K(I) \left[c_K\cdot \frac{X}{N(I)}\right].
    \end{align}
    We note here, by \textit{Theorem 1.5}, that for a fixed ideal $I \subset \mathcal{O}_K$,
    \begin{align}
       \# \{J\subset \mathcal{O}_K:J\subset I; \;N(J)\leq X\} = \sum_{\substack{J \subset \mathcal{O}_K \\ N(J)\leq X \\ I \supset J  }}1 = c_K\cdot \frac{X}{N(I)} + O\left(\left(\frac{X}{N(I)}\right)^{1-\frac{1}{d}}\right) .
    \end{align}
Thus, using (4.27) in (4.26), we get
    \begin{align}
       X\sum_{\substack{2\leq N(I) \leq X \\ I \in S(L/K;C)}}\frac{\mu_K(I)\omega_K(I)}{N(I)} &= \frac{1}{c_K}\sum_{\substack{2\leq N(I) \leq X \\ I \in S(L/K;C)}} \mu_K(I)\omega_K(I) \left[\sum_{\substack{J \subset \mathcal{O}_K \\ N(J)\leq X \\ I \supset J  }}1 + O\left(\left(\frac{X}{N(I)}\right)^{1-\frac{1}{d}}\right) \right] \nonumber \\
       &= \frac{1}{c_K}\sum_{\substack{2\leq N(I) \leq X \\ I \in S(L/K;C)}} \mu_K(I)\omega_K(I) \sum_{\substack{J \subset \mathcal{O}_K \\ N(J)\leq X \\ I \supset J  }}1 + O\left(  \sum_{\substack{2\leq N(I) \leq X \\ I \in S(L/K;C)}} \mu_K(I)\omega_K(I)\left(\frac{X}{N(I)}\right)^{1-\frac{1}{d}} \right) .
    \end{align}
    Let us first estimate the error term on the right of (4.28). To do so, we recall the following bound of the sum of the generalized Möbius function (see \textit{Lemma 2.2 (ii)} in [SW18]) that follows analogously from their classical counterparts:
    \begin{align}
          \sum_{N(I)\leq X} \frac{\mu_K(I)}{N(I)}&=O\left(e^{-c_6(\log X)^{\frac{1}{12}}}\right).
    \end{align}
    Also, we use the following number field analogue [Na84] of the Hardy-Ramanujan estimate of $\omega(n)$ [HW79]:
    \begin{align}
        \omega_K(I)\sim \log\log N(I)  .
    \end{align}
    Hence, using (4.29) and (4.30) in the error in (4.28), we have
    \begin{align}
        \sum_{\substack{2\leq N(I) \leq X \\ I \in S(L/K;C)}} \mu_K(I)\omega_K(I)\left(\frac{X}{N(I)}\right)^{1-\frac{1}{d}} &\ll \log\log X \sum_{N(I)\leq X}\frac{\mu_K(I)}{N(I)}X^{1-\frac{1}{d}}N(I)^{\frac{1}{d}} \nonumber \\
        &\ll \frac{X\log\log X}{e^{c_6(\log X)^{\frac{1}{12}}}}.
    \end{align}
    Now, for the main term in the RHS of (4.28), we observe with an interchange in summation that
    \begin{align}
        \frac{1}{c_K}\sum_{\substack{2\leq N(I) \leq X \\ I \in S(L/K;C)}} \mu_K(I)\omega_K(I) \sum_{\substack{J \subset \mathcal{O}_K \\ N(J)\leq X \\ I \supset J  }}1 &=\frac{1}{c_K}   \sum_{2\leq N(J)\leq X}\sum_{I\supset J } \mu_K(I)\omega_K(I)f(I) .
    \end{align}
    We first split the inner sum in the RHS of (4.32) into the following two parts:
    \begin{align}
        \sum_{\substack{I\supset J}} \mu_K(I)\omega_K(I)f(I)  = \sum_{\substack{I\supset J}} \mu_K(I)(\omega_K(I)-1)f(I) +\sum_{\substack{I\supset J}} \mu_K(I)f(I) .
    \end{align}
    Hence, (4.32) and (4.33) together yield 
    \begin{align}
        \frac{1}{c_K}   \sum_{2\leq N(J)\leq X}\sum_{\substack{I\supset J}} \mu_K(I)\omega_K(I)f(I)  = \frac{1}{c_K}   \sum_{2\leq N(J)\leq X} \sum_{\substack{I\supset J}} \mu_K(I)(\omega_K(I)-1)f(I)  + \frac{1}{c_K}   \sum_{2\leq N(J)\leq X}\sum_{\substack{I\supset J}} \mu_K(I)f(I)  .
    \end{align}
    Applying \textit{Theorem 2.3}, and then \textit{Lemma 3.1} and \textit{Theorem 4.1} together in the first sum in the RHS of (4.34), we have
    \begin{align}
        \frac{1}{c_K}   \sum_{2\leq N(J)\leq X} \sum_{\substack{I\supset J}} \mu_K(I)(\omega_K(I)-1)f(I) &= \frac{1}{c_K} \sum_{\substack{2\leq N (J)\leq X \\ Q^1(J)=1}}Q_C^2(J) \nonumber \\
        &=  \frac{|C|}{|G|}\cdot X + O\left(\frac{X(\log\log X)^2}{\log X}\right).
    \end{align}
    Further, applying \textit{Lemma 2.2} and then \textit{Theorem 1.1} of [SW18] in the second sum in the RHS of (4.34), we have
    \begin{align}
        \frac{1}{c_K}   \sum_{2\leq N(J)\leq X}\sum_{\substack{I\supset J}} \mu_K(I)f(I)  &= -\frac{1}{c_K}   \sum_{2\leq N(J)\leq X} Q_C^1(J)  \nonumber \\
        &= -\frac{|C|}{|G|}\cdot X +O\left(\frac{X}{e^{c_5(\log X)^{\frac{1}{3}}}}\right).
    \end{align}
    Therefore, combining (4.32), (4.34), (4.35), and (4.36), we get
    \begin{align}
         \frac{1}{c_K}\sum_{\substack{2\leq N(I) \leq X \\ I \in S(L/K;C)}} \mu_K(I)\omega_K(I) \sum_{\substack{J \subset \mathcal{O}_K \\ N(J)\leq X \\ I \supset J  }}1 \ll \frac{X(\log\log X)^2}{\log X}.
    \end{align}
    Thus, (4.28) along with (4.31) and (4.37) yield
    \begin{align}
         X\sum_{\substack{2\leq N(I) \leq X \\ I \in S(L/K;C)}}\frac{\mu_K(I)\omega_K(I)}{N(I)}  \ll \frac{X(\log\log X)^2}{(\log X)}.
    \end{align}
    Canceling $X$ on both sides in (4.38), we get
    \begin{align}
        \sum_{\substack{2\leq N(I) \leq X \\ I \in S(L/K;C)}}\frac{\mu_K(I)\omega_K(I)}{N(I)}  \ll \frac{(\log\log X)^2}{(\log X)}.
    \end{align}
Further, using the quantitative version of \textit{Theorem 1.2} of [SW18], we have from (4.39) that
\begin{align}
    O\left(\frac{(\log\log X)^2}{\log X}\right) =  \sum_{\substack{2\leq N(I) \leq X \\ I \in S(L/K;C)}}\frac{\mu_K(I)\omega_K(I)}{N(I)}  &=  \sum_{\substack{2\leq N(I) \leq X \\ I \in S(L/K;C)}}\frac{\mu_K(I)(\omega_K(I)-1)}{N(I)} + \sum_{\substack{2\leq N(I) \leq X \\ I \in S(L/K;C)}}\frac{\mu_K(I)}{N(I)} \nonumber \\
    &=\sum_{\substack{2\leq N(I) \leq X \\ I \in S(L/K;C)}}\frac{\mu_K(I)(\omega_K(I)-1)}{N(I)} - \frac{|C|}{|G|} + O\left(e^{-c_6(\log X)^{\frac{1}{12}}}\right).
\end{align}
Thus, rearranging (4.40), finally get our desired result:
\begin{align}
    \sum_{\substack{2\leq N(I) \leq X \\ I \in S(L/K;C)}}\frac{\mu_K(I)(\omega_K(I)-1)}{N(I)} = \frac{|C|}{|G|}+ O\left(\frac{(\log\log X)^2}{\log X}\right) .
\end{align}
Taking $X\to \infty$ on both sides of (4.41), we have our desired formula for the Chebotarev Density.
\end{proof}
\begin{remark}
    It is important to note here that this method to obtain the formula is significantly different from any that has been used earlier. The method used in [SW18] uses the sums of $\mu_K(I)\log N(I)$, which faces difficulties in our case as the error estimates grow large due to the largeness of $\Psi_{K,2}(X,Y)$. Difficulties were also faced in following a method similar to the one used in [Se25] due to the absence of a one-to-one correspondence between the multiples of an ideal $I \subset \mathcal{O}_K$ and multiples of its norm  $N(I)$. 
\end{remark}

 \section{Two Estimates of Sums involving $\mu_K(I)$ and $\omega_K(I)$}
 In [Al77], [AJ24] and [Se25], we observe that the authors have used the estimates of sums involving $\mu(n)$ and $\omega(n)$, which are required to prove the respective final results. As can be observed in \S4, such estimates were not required to obtain our main result. Yet, we are capable of finding estimates of such sums with the machinery that we have developed. In this section, therefore, we will prove two theorems that give us quantitative estimates of such sums, which, although they are not necessary to get the formula for the Chebotarev Density, can be viewed as an application of our duality identities.
 \begin{theorem}
     For a finite Galois extension $L/K$ and a fixed conjugacy class $C\subset G=Gal(L/K)$, we have for some positive constant $c_{10}$ that
     \begin{align}
         \sum_{\substack{N(I)\leq X \\ I \in S(L/K;C)}}\mu_K(I) \ll \frac{X\log X}{e^{c_{10}(\log X)^{\frac{1}{12}}}}. \nonumber
     \end{align}
 \end{theorem}
 \begin{proof}
     Apriori, we note that the function $f$ has already been defined in \S2 (see \textit{Lemma 2.2}) and \S4 (see \textit{Theorem 4.3}) as the indicator function $S(L/K;C)$. Therefore, continuing from the LHS above, we have
     \begin{align}
          \sum_{\substack{N(I)\leq X \\ Ic \in S(L/K;C)}}\mu_K(I) = \sum_{N(I)\leq X}\mu_K(I)f(I).
     \end{align}
     Now, using the number field analogue of M\"obius Inversion in \textit{Lemma 2.2}, we get from (5.1)
     \begin{align}
         \sum_{N(I)\leq X}\mu_K(I)f(I) = -\sum_{N(I)\leq X}\sum_{JJ'=I}\mu_K(J')Q_C^1(J).
     \end{align}
     Note, the second sum runs over all ideals $J \supset I$, and we write $Q_C(J)=Q_C^1(J)$. Truncating the double sum at $\sqrt{X}$ and using the hyperbola method, we deduce
     \begin{align}
          \sum_{N(I)\leq X}\mu_K(I)f(I) = - \sum_{N(J')\leq \sqrt{X}}\mu_K(J')\sum_{N(J)\leq \frac{X}{N(J')}}Q_C^1(J) -\sum_{N(J)\leq \sqrt{X}}Q_C^1(J)\sum_{\sqrt{X}<N(J')\leq \frac{X}{N(J)}}\mu_K(J').
     \end{align}
     Using \textit{Theorem 1.1} in [SW18] in the first double sum on the right of (5.3), we get
     \begin{align}
         \sum_{N(J')\leq \sqrt{X}}\mu_K(J')\sum_{N(J)\leq \frac{X}{N(J')}}Q_C^1(J) = \sum_{N(J')\leq \sqrt{X}}\mu_K(J')\left[c_K \cdot \frac{|C|}{|G|}\cdot \frac{X}{N(J')}+O\left(\frac{X}{N(J')e^{c_5\left(\log \frac{X}{N(J')}\right)^{\frac{1}{3}}}}\right) \right].
     \end{align}
     Thus, using (4.29) in (5.4), we get that
     \begin{align}
           \sum_{N(J')\leq \sqrt{X}}\mu_K(J')\sum_{N(J)\leq \frac{X}{N(J')}}Q_C^1(J) &\ll \frac{X}{e^{c_6(\log X)^{\frac{1}{12}}}} + \frac{X\log X}{e^{c_7(\log X)^{\frac{1}{3}}}} \nonumber \\
           &\ll \frac{X\log X}{e^{c_8(\log X)^{\frac{1}{12}}}},
     \end{align}
     where $c_8=\min\{c_9,c_7\}$ and $c_7=\frac{c_5}{\sqrt[3]{2}},\;c_9 = \frac{c_6}{\sqrt[12]{2}}$. Here, we note another\footnote[11]{first one is (4.29) in \S4} estimate of the sum of the generalized M\"obius function (see \textit{Lemma 2.2 (i)} in [SW18]) that follow analogously from its classical case:
     \begin{align}
         \sum_{N(I)\leq X} \mu_K(I)&=O\left(Xe^{-c_6(\log X)^{\frac{1}{12}}}\right)  .
     \end{align}
     For the second double sum on the right of (5.3), we observe using (5.6) that
\begin{align}
    \sum_{N(J)\leq \sqrt{X}}Q_C^1(J)\sum_{\sqrt{X}<N(J')\leq \frac{X}{N(J)}}\mu_K(J') \ll \sum_{N(J)\leq \sqrt{X}} \frac{X}{N(J)e^{c_6\left(\log \frac{X}{N(J)}\right)^{\frac{1}{12}}}} \ll \frac{X\log X}{e^{c_9(\log X)^{\frac{1}{12}}}}.
\end{align}
Here, we use the fact that $Q_C^1(I)$ is bounded by the degree of extension $K/\mathbb{Q}$.  Therefore, combining (5.3), (5.5), and (5.7), and choosing $c_{10} = \min\{c_8,c_9\}$, we have
    \begin{align}
          \sum_{N(I)\leq X}\mu_K(I)f(I) \ll \frac{X\log X}{e^{c_{10}(\log X)^{\frac{1}{12}}}}.
    \end{align}
    This proves our lemma. 
 \end{proof}
 \begin{remark}
     We observe that here we used the first order duality lemma. The proofs in [SW18] did not require such an estimate and therefore, \textit{Theorem 5.1} was not proved then. 
 \end{remark}
\noindent The next theorem gives us an estimate of the sum involving both $\mu_K(I)$ and $\omega_K(I)$. In this theorem, we will use the second order duality to get our required estimate.
\begin{theorem}
    For a finite Galois extension $L/K$ and a fixed conjugacy class $C\subset G=Gal(L/K)$, we have 
    \begin{align}
        \sum_{\substack{N(I)\leq X \\ I \in S(L/K;C)}}\mu_K(I)\omega_K(I) \ll \frac{X(\log\log X)^4}{\log X}. \nonumber
    \end{align}
\end{theorem}
\begin{proof}
    Let $f(I)$ be the characteristic function of $S(L/K;C)$. Therefore, we have
    \begin{align}
          \sum_{\substack{N(I)\leq X \\ I \in S(L/K;C)}}\mu_K(I)\omega_K(I) &= \sum_{N(I)\leq X }\mu_K(I)\omega_K(I)f(I) \nonumber \\
          &= \sum_{N(I)\leq X }\mu_K(I)(\omega_K(I)-1)f(I) + \sum_{N(I)\leq X }\mu_K(I)f(I).
    \end{align}
    Using M\"obius Inversion on \textit{Theorem 2.3}  and then applying the same in the first sum on the right of (5.9), we have 
    \begin{align}
        \sum_{N(I)\leq X }\mu_K(I)(\omega_K(I)-1)f(I) = \sum_{N(I)\leq X}\sum_{\substack{JJ' =I \\ Q^1(J)=1}} \mu_K(J')Q_C^2(J) .
    \end{align}
    Using the hyperbola method by truncating the double sum in (5.10) at $T$\footnote[12]{a suitable choice of $T$ will be made at the end of the proof}, we get
    \begin{align}
          \sum_{N(I)\leq X }\mu_K(I)(\omega_K(I)-1)f(I) = \sum_{N(J')
          \leq T}\mu_K(J')\sum_{\substack{N(J)\leq \frac{X}{N(J')}\\ Q^1(J)=1}}Q_C^2(J)+\sum_{\substack{N(J)\leq \frac{X}{T} \\ Q^1(J)=1 }}Q_C^2(J)\sum_{T<N(J')\leq \frac{X}{N(J)}}\mu_K(J').
    \end{align}
    Using \textit{Theorem 4.1} and the estimate in (4.29), we get from the first double sum in (5.11) that
    \begin{align}
        \sum_{N(J')
          \leq T}\mu_K(J')\sum_{\substack{N(J)\leq \frac{X}{N(J')} \\ Q^1(J)=1}}Q_C^2(J) &= \sum_{N(J')
          \leq T}\mu_K(J')\left[ c_K\cdot \frac{|C|}{|G|}\cdot 
        \frac{X}{N(J')}+ O\left(\frac{X(\log\log X)^2}{\log 
        \left(\frac{X}{N(J')}\right)}\right)  \right] \nonumber \\
        &\ll c_K\cdot \frac{|C|}{|G|}\cdot X\sum_{N(J')\leq T} \frac{\mu_K(J')}{N(J')} + O\left(\frac{X(\log \log X)^2}{\log \left(\frac{X}{T}\right)}\sum_{N(J')\leq T}\frac{1}{N(J')}\right)
        \nonumber \\
        &\ll \frac{X}{e^{c_6(\log T)^{\frac{1}{12}}}} + \frac{X\log T(\log\log X)^2}{\log \left(\frac{X}{T}\right)}.
    \end{align}
    For the second double sum in the RHS of (5.11), we use (5.6) and the fact that $Q_C^2(I)$ is bounded to get
    \begin{align}
        \sum_{\substack{N(J)\leq \frac{X}{T} \\ Q^1(J)=1}}Q_C^2(J)\sum_{T<N(J')\leq \frac{X}{N(J)}}\mu_K(J') \ll \sum_{N(J)\leq \frac{X}{T}} \frac{X}{N(J')e^{c_6\left(\log \frac{X}{N(J')}\right)^{\frac{1}{12}}}} \ll \frac{X\log X}{e^{c_6(\log T)^{\frac{1}{12}}}}.
    \end{align}
    Further, in the second sum of (5.9), we use \textit{Theorem 5.1} to get
    \begin{align}
        \sum_{N(I)\leq X }\mu_K(I)f(I) \ll \frac{X\log X}{e^{c_8(\log X)^{\frac{1}{12}}}} .
    \end{align}
    Finally, combining (5.9), (5.11), (5.12), (5.13), and (5.14), we have
    \begin{align}
         \sum_{\substack{N(I)\leq X \\ I \in S(L/K;C)}}\mu_K(I)\omega_K(I) \ll  \frac{X}{e^{c_6(\log T)^{\frac{1}{12}}}} + \frac{X\log T(\log\log X)^2}{\log \left(\frac{X}{T}\right)} + \frac{X\log X}{e^{c_6(\log T)^{\frac{1}{12}}}} +\frac{X\log X}{e^{c_8(\log X)^{\frac{1}{12}}}} .
    \end{align}
    Therefore, with the choice of $T=e^{(2\log\log X)^2}$ above, we get our desired result.
    \end{proof}
    \begin{remark}
        Both the above theorems are important applications of our duality identities. Not only does the duality help us get a new formula for the Chebotarev Density, but it also helps us estimate different sums over restricted sets of ideals. Higher order dualities can similarly be used to deduce such estimates of other sums.\footnote[13]{one can follow methods in [AS(ip)] to get estimates of new sums} 
    \end{remark}
\section{Duality for Arbitrary Sets of Prime Ideals}
In \S2, we prove a general higher order duality identity (\textit{Theorem 2.4}) involving prime ideals of $k^{th}$ largest norm and the smallest norm. But the duality has been specific to only a particular condition, i.e., for prime ideals being unramified in the bigger field of extension and satisfying the Artin Symbol condition for a fixed conjugacy class $C$ of a Galois group of a corresponding Galois extension. But we can extend the duality identity for a higher class of conditions without having to hardly change anything in the proofs. In this section, we will state\footnote[14]{the proof will be omitted as they exactly follow the same idea as \textit{Theorem 2.3} and \textit{Theorem 2.4}} the general duality for arbitrary but fixed sets of prime ideals in a given ring of integers. \\ \\
Let $K $ be a number field and $\mathcal{O}_K$ the corresponding ring of integers. Let $\mathcal{A}\subset \mathfrak{P}(K)$\footnote[15]{see p. 6 for definition} be an arbitrary but fixed set of prime ideals in $\mathcal{O}_K$. We define the following sets:
\begin{align}
    S(\mathcal{A})&:=\{I\subset\mathcal{O}_K:I \text{ is salient with unique }\mathfrak{p};\;\mathfrak{p}\in\mathcal{A}\} ,\nonumber \\
    Q_{\mathcal{A}}^k(I)&:=\#\{\mathfrak{p}\subset \mathcal{O}_K:\mathfrak{p}\;\text{is a prime};\;I\subset\mathfrak{p};\;M_k(I)=N(\mathfrak{p});\;\mathfrak{p}\in\mathcal{A}\} .\nonumber
\end{align}
\begin{theorem}
      Let $K$ be a number field and $k$ be a positive integer. Let $I \subset \mathcal{O}_K$ be an ideal such that $Q^i(I)=1$, for $1 \leq i\leq k-1$. Let $f$ be the characteristic function of the set $S(\mathcal{A})$. Then the following identity holds:
    \begin{align}
        \sum_{J \supset I}\mu_K(J){{\omega_K(J)-1} \choose k-1}f(J) = (-1)^kQ_{\mathcal{A}}^k(I). \nonumber
    \end{align}
\end{theorem}
\begin{remark}
    The proof of \textit{Theorem 6.1} completely replicates that of \textit{Theorem 2.4}.
\end{remark} 
\noindent We can now use \textit{Theorem 6.1} to obtain density type results, similar to our \textit{Theorem 4.3} above. We will, in our hypothesis, assume that the averages of the sets $Q_{\mathcal{A}}^k(I)$ exist. We state such a theorem for $k=2$ below.
\begin{theorem}
    Let $\mathcal{A}$ be a set of prime ideals in $\mathcal{O}_K$. Let $f$ be the characteristic function of $S(\mathcal{A})$. If 
    \begin{align}
        \sum_{2\leq N(I)\leq X}Q_{\mathcal{A}}^1(I) \sim \mathcal{K}\cdot X \quad and\quad \sum_{2\leq N(I)\leq X}Q_{\mathcal{A}}^2(I) \sim \mathcal{K}\cdot X ,\nonumber
    \end{align}
    then
    \begin{align}
        \lim_{x\to \infty}\sum_{\substack{2 \leq N(I)\leq X \\ I\in S(\mathcal{A})}} \frac{\mu_K(I)\omega_K(I)}{N(I)} = 0.\nonumber
    \end{align}
\end{theorem}
 \begin{proof}
     This proof also replicates that of \textit{Theorem 4.3}, but we will provide a small sketch that reflects the idea. We start by considering
     \begin{align}
        X\sum_{\substack{2\leq N(I) \leq X \\ I \in S(\mathcal{A})}}\frac{\mu_K(I)\omega_K(I)}{N(I)} 
        &= \frac{1}{c_K}\sum_{\substack{2\leq N(I) \leq X \\ I \in S(\mathcal{A})}} \mu_K(I)\omega_K(I) \left[c_K\cdot \frac{X}{N(I)}\right] \nonumber \\
       &= \frac{1}{c_K}\sum_{\substack{2\leq N(I) \leq X \\ I \in S(\mathcal{A})}} \mu_K(I)\omega_K(I) \sum_{\substack{J \subset \mathcal{O}_K \\ N(J)\leq X \\ I \supset J  }}1 + O\left(  \sum_{\substack{2\leq N(I) \leq X \\ I \in S(\mathcal{A})}} \mu_K(I)\omega_K(I)\left(\frac{X}{N(I)}\right)^{1-\frac{1}{d}} \right) .
    \end{align}
    To estimate the error, we use (4.29) and (4.30) to get
    \begin{align}
        \sum_{\substack{2\leq N(I) \leq X \\ I \in S(\mathcal{A})}} \mu_K(I)\omega_K(I)\left(\frac{X}{N(I)}\right)^{1-\frac{1}{d}} \ll \frac{X\log\log X}{e^{c_6(\log X)^{\frac{1}{12}}}}.
    \end{align}
    Now, for the main term in the RHS of (6.1), we observe that
    \begin{align}
        \frac{1}{c_K}\sum_{\substack{2\leq N(I) \leq X \\ I \in S(\mathcal{A})}} \mu_K(I)\omega_K(I) \sum_{\substack{J \subset \mathcal{O}_K \\ N(J)\leq X \\ I \supset J  }}1 &=\frac{1}{c_K}   \sum_{2\leq N(J)\leq X}\sum_{I\supset J } \mu_K(I)\omega_K(I)f(I) \nonumber \\
            & = \frac{1}{c_K}   \sum_{2\leq N(J)\leq X} \sum_{\substack{I\supset J}} \mu_K(I)(\omega_K(I)-1)f(I)  + \frac{1}{c_K}   \sum_{2\leq N(J)\leq X}\sum_{\substack{I\supset J}} \mu_K(I)f(I)   .        
    \end{align}
    Applying \textit{Theorem 6.1} for $k=2$ and then \textit{Lemma 3.1} in the first sum in the RHS of (6.3), we have
    \begin{align}
        \frac{1}{c_K}   \sum_{2\leq N(J)\leq X} \sum_{\substack{I\supset J}} \mu_K(I)(\omega_K(I)-1)f(I) &= \frac{1}{c_K} \sum_{\substack{2\leq N (J)\leq X \\ Q^1(J)=1}}Q_C^2(J) =  \frac{1}{c_K}\cdot\mathcal{K}\cdot X + o(X),
  \end{align}
    and applying \textit{Theorem 6.1} for $k=1$ in the second in the RHS of (6.3), we have
    \begin{align}
        \frac{1}{c_K}   \sum_{2\leq N(J)\leq X}\sum_{\substack{I\supset J}} \mu_K(I)f(I)  &= -\frac{1}{c_K}   \sum_{2\leq N(J)\leq X} Q_C^1(I)  = -\frac{1}{c_K}\cdot \mathcal{K}\cdot X +o(X).
    \end{align}
    Thus, combining everything above, we have
    \begin{align}
         X\sum_{\substack{2\leq N(I) \leq X \\ I \in S(\mathcal{A})}}\frac{\mu_K(I)\omega_K(I)}{N(I)} =o(X) \implies \sum_{\substack{2\leq N(I) \leq X \\ I \in S(\mathcal{A})}}\frac{\mu_K(I)\omega_K(I)}{N(I)} = o(1) .
    \end{align}
   Thus, rearranging the sum and using the hypothesis again gives us our required result.
\end{proof}
\noindent We can prove the following corollary too in the similar fashion as of \textit{Theorem 6.3}.
\begin{corollary}
       Let $\mathcal{A}$ be a set of prime ideals in $\mathcal{O}_K$. Let $f$ be the characteristic function of $S(\mathcal{A})$. If 
    \begin{align}
        \sum_{2\leq N(I)\leq X}Q_{\mathcal{A}}^1(I) \sim \mathcal{K}\cdot X \quad and\quad \sum_{2\leq N(I)\leq X}Q_{\mathcal{A}}^2(I) \sim \mathcal{K}\cdot X ,\nonumber
    \end{align}
    then
    \begin{align}
        \lim_{x\to \infty}\sum_{\substack{2 \leq N(I)\leq X \\ I\in S(\mathcal{A})}} \frac{\mu_K(I)(\omega_K(I)-1)}{N(I)} = \frac{1}{c_K}\cdot \mathcal{K}.\nonumber
    \end{align}
\end{corollary}
 \begin{remark}
     Choosing the set $\mathcal{A}$ to be the prime ideals that are unramified in $L$ such that their Artin symbols are all $C$, where $L/K$ is a Galois extension and $C\subset G=Gal(L/K)$, we see that $\mathcal{K}=c_K\cdot \frac{|C|}{|G|}$, and hence, \textit{Theorem 4.3} follows qualitatively from \textit{Corollary 6.3.1}.
 \end{remark}

 \noindent It is always a challenge to prove the asymptotics in the hypothesis of \textit{Theorem 6.3}. In our case, \textit{Theorem 4.1} and \textit{Theorem 1.1} in [SW18] are extremely important results, and  \textit{Theorem 4.3} is a consequence of both of these results. It will be worthwhile to study such collections $\mathcal{A}$ of prime ideals in given rings of integers that satisfy such asymptotics and how the constant $\mathcal{K}$ differs in each case. Of course, if the constants are different, say $\mathcal{K}_1,\mathcal{K}_2$ respectively for $k=1,2$, we will have
 \begin{align}
     \lim_{x\to \infty}\sum_{\substack{2\leq N(I) \leq X \\ I \in S(\mathcal{A})}}\frac{\mu_K(I)\omega_K(I)}{N(I)} = \frac{1}{c_K}(\mathcal{K}_2-\mathcal{K}_1). \nonumber 
 \end{align}
There have been works by Kural, McDonald, and Sah [KMS20] and Wang [Wa21] that extends works of Alladi, Dawsey, Sweeting, and Woo. It is hoped that this general duality and the new approach to obtaining formulas for the Chebotarev Density can now open up even more general extensions of results already proved by the authors mentioned above. Sahoo and Jha [JS(ip)] have also worked out the function field analogue of the Alladi duality. A function field analogue of this higher order duality between prime ideals will lead to a complete duality identity for global fields, which might even lead to further exploration.  

\vspace{1cm}
\noindent \textbf{\textit{Acknowledgements:}} I am indebted to my Doctoral Advisor, Prof. K. Alladi, for his support and guidance throughout this work. I also sincerely thank Prof. A. Vatwani and her PhD students, Mr. J. Sahoo and Mr. P. N. Jha, for helping me out with new explanations and valuable insights that led to the completion of this work.
\section{References}
\begin{enumerate}
    \item[{[Al77]}] Alladi, K., ``Duality between prime factors and an application to the prime number theorem for arithmetic progressions", \textit{J. Num. Th.}, \textbf{9} (1977), 436-451.
    \item[{[AJ24]}] Alladi, K. and Johnson, J., ``Duality between prime factors and the prime number theorem for arithmetic progressions - II", (\textit{submitted pre-print}), 2024. \href{https://nam10.safelinks.protection.outlook.com/?url=http%3A%2F%2Farxiv.org%2Fabs%2F2410.18259&data=05%7C02%7Csengupta.s%40ufl.edu%7Cbb6e0dfda02940b30dbb08dcf7758568%7C0d4da0f84a314d76ace60a62331e1b84%7C0%7C0%7C638657331341043920%7CUnknown%7CTWFpbGZsb3d8eyJWIjoiMC4wLjAwMDAiLCJQIjoiV2luMzIiLCJBTiI6Ik1haWwiLCJXVCI6Mn0%3D%7C0%7C%7C%7C&sdata=YQy1HVSyFY7D%2F13FeVYxRe5BIEnZ%2BmAsVQuJ4xLqiOY%3D&reserved=0}{arXiv:2410.18259} 
    \item[{[AS(ip)]}] Alladi, K. and Sengupta, S., ``Higher order duality between prime factors and primes in arithmetic progressions", (\textit{in preparation}). 
    \item[{[Br51]}]  de Bruijn, N. G., ``On the number of positive integers $\leq x$ and free of prime factors $>y$", \textit{Indag. Math.}, \textbf{13} (1951), 50-60.
    \item[{[Da15]}] Dawsey, M.L., ``A new formula for Chebotarev densities", \textit{Res. in Num. Th.}, \textbf{3} (2017), 1-13. 
    \item[{[De61]}] Delange, H., ``Sur les fonctions arithm\'etiques multiplicatives", \textit{Ann. Sci. Ecole Norm. Sup.}, \textbf{78} (1961), 273-304 
    \item[{[HW79]}] Hardy, G. H., and Wright, E. M., ``An introduction to the theory of numbers", \textit{Oxford university press}, (1979) 
    \item[{[Hi86]}] Hildebrand, A., ``On the number of positive integers $\leq x$ and free of prime factors $>y$", \textit{J. Num. Th.} \textbf{22} (1986), 289-307 
    \item[{[JS(ip)]}] Jha, P. N. and Sahoo, J., ``Higher order dualities and Alladi-Johnson identities over function fields", \textit{(in preparation)} 
    \item[{[Kr90]}] Krause, U., ``Anzahl der Ideale $\mathfrak{a}$ mit $N\mathfrak{a}\leq x$ und Primteilern $\mathfrak{p}$ mit $N\mathfrak{p}\leq y$", \textit{Diplomarbeit}, Phillipe-Universit\"at Marburg, (1989)  
\item[{[KMS20]}] Kural, M., McDonald, V. and Sah, A.,
``Möbius formulas for densities of sets of prime ideals", 
\textit{Arch. Math. (Basel)} \textbf{115} (2020), 53–66.
\item[{[LO77]}]  Lagarias, J. C., and Odlyzko, A. M., ``Effective versions of the Chebotarev density theorem", \textit{Algebraic Number Fields: L-functions and Galois properties (A. Fr\"ohlich, Acad. Press. London)}, \textbf{7} 
\item[{[LaT99]}] Landau, E., ``Neuer Beweis der Gleichung $\sum_{k=1}^{\infty}\frac{\mu(k)}{k}=0$", \textit{Inaugural-Dissertation}, Berlin (1899)
\item[{[La03]}] Landau, E., ``Neuer Beweis des Primzahlsatzes und Beweis des Primidealsatzes", \textit{Math. Ann.} \textbf{56}(4) (1903), 645-670
\item[{[Le23]}] Lee, E. S., ``Explicit Merten's Theorems for Number Fields", \textit{Bul. Aus. Math. Soc.}, \textbf{108}(1) (2023), 169-172.
\item[{[Ma(up)]}] Maier, H., ``On integers free of large prime divisors", \textit{unpublished} 
\item[{[MV07]}] Montgomery, H. L. and Vaughan, R. C., ``Multiplicative Number Theory I. Classical Theory", \textit{Cam. Adv. Math.}, \textbf{97} (2007) 
\item[{[Mo92]}] Moree, P., ``An interval result for the number field $\Psi(x,y)$ function", \textit{Manuscripta Mathematica}, \textbf{76} (1992), 437-450. 
\item[{[MO07]}] Murty, M. R. and Order, J. V., ``Counting integral ideals in a number field", \textit{Expos. Math.} \textbf{25}(1) (2007), 53-66 
\item[{[Na84]}] Narkiewicz, W., ``Number theory", \textit{World Scientific} (1984).
\item[{[No71]}] Norton, K. K., ``Numbers with Small Prime Factors, and the Least $k^{th}$ Power Nonresidue", \textit{Amer. Math. Soc. Prov.} (1971) 
\item[{[Se25]}] Sengupta, S., ``Algebraic analogues of theorems of Alladi-Johnson using the Chebotarev Density Theorem'', {\it{Research Num. Th.}}, {\bf{11}} (2025) Article 44. 
\item[{[SW18]}]  Sweeting, N. and Woo, K., ``Formulas for Chebotarev densities of Galois extensions of number fields", \textit{Res. in Num. Th.}, \textbf{5} (2018), 1-13. 
\item[{[Te00]}] Tenenbaum, G., ``A rate estimate in Billingsley's theorem for the size distribution of large prime factors", \textit{Quart. J. Math.} \textbf{51} (2000), 385-403. 
 \item[{[Ts26]}] Tschebotareff, N., ``Die Bestimmung der Dichtigkeit einer Menge von Primzahlen, welche zu einer gegebenen Substitutionsklasse geh\"oren", \textit{Math. Ann.} \textbf{95} (1926), 191-228 
\item[{[Wa21]}] Wang, B., ``Analogues of Alladi's formula", \textit{J. Number Theory} \textbf{221} (2021), 232–246. 
\end{enumerate}

\end{document}